\documentclass[10pt]{article}

\usepackage{hyperref}

\usepackage{titlesec}

\usepackage{amsmath}

\usepackage{epsfig, amssymb, amsfonts, dsfont}
\usepackage{amsthm}
\usepackage{amstext}
\usepackage{amsopn}
\usepackage{mathrsfs}
\usepackage{subfigure}
\usepackage{graphicx}
\usepackage[left=2.3cm,top=2cm,right=2.3cm,bottom=2cm, nohead,foot=1cm]{geometry}
\usepackage[makeroom]{cancel}
\usepackage{color}
\usepackage{enumerate}
\usepackage{comment}
\usepackage{stackrel}
\usepackage{stmaryrd}
\usepackage{mathtools}
\usepackage{amsbsy}

\usepackage{tikz}
\usepackage{tikz-cd}

\SetSymbolFont{stmry}{bold}{U}{stmry}{m}{n}

\numberwithin{equation}{section}

\newtheorem{theorem}{Theorem}[section]

\newtheorem{lemma}[theorem]{Lemma}

\newtheorem{corollary}[theorem]{Corollary}
\newtheorem{proposition}[theorem]{Proposition}

\theoremstyle{definition}

\newtheorem{remark}[theorem]{Remark}

\newcommand{\R}{{\mathbb R}}

\newcommand{\Hess}{{\mathbf{H}}}

\DeclareFontFamily{U}{mathx}{\hyphenchar\font45}
\DeclareFontShape{U}{mathx}{m}{n}{
      <5> <6> <7> <8> <9> <10>
      <10.95> <12> <14.4> <17.28> <20.74> <24.88>
      mathx10
      }{}
\DeclareSymbolFont{mathx}{U}{mathx}{m}{n}
\DeclareMathSymbol{\bigtimes}{1}{mathx}{"91}

\usepackage{relsize}

\newcommand{\Tr}{{\textup{Tr}}}

\title{Pathwise structure of the three-dimensional attractive one-point interaction diffusion}

\date{  }

 \author{\textbf{Barkat Mian}\footnote{ {\tt bmian@utk.edu}} \vspace{.1cm}  \\  University of Tennessee Knoxville, Department of Mathematics   }

\begin{document}
\maketitle

\begin{abstract}
We study the pathwise behavior of the three-dimensional attractive one-point interaction diffusion whose law was constructed in~\cite{CranstonKoralovMolchanovVainberg2}, corresponding to the singular Schr\"odinger Hamiltonian
\[
\frac12\Delta+\frac{\beta}{2}\delta_0,
\qquad \beta>0.
\]
We identify a local stochastic differential equation satisfied by the process away from the origin and use it to construct a natural submartingale whose increasing component in the Doob-Meyer decomposition is supported on the set of times at which the process visits the origin. In particular, we show that the process visits the origin with positive probability and that the law conditioned on avoiding the origin is three-dimensional Wiener measure.
\end{abstract}

\vspace{.2cm}

\section{Introduction}
A three-dimensional heat equation with an attractive one-point
potential at the origin is formally given by
\[
\partial_t u
=
\frac12\Delta u+\frac{\beta}{2}\delta_0 u,
\]
where \(\Delta\) denotes the Laplacian on \(\mathbb R^3\), \(\delta_0\) denotes the Dirac mass at the origin, and \(\beta>0\) is the coupling constant. Equivalently, the corresponding three-dimensional Schr\"odinger Hamiltonian with a one-point potential at the origin is formally given by
\[
\frac12\Delta+\frac{\beta}{2}\delta_0,
\]
and it can be understood as a self-adjoint extension of the Laplacian restricted to \(C_0^\infty(\mathbb R^3\setminus\{0\})\). 
Alternatively, for each $\beta>0$, a one-point potential Schr\"odinger operator \(\Delta^{\beta}\) can be obtained as a norm resolvent limit of regularized Schr\"odinger operators of the form
\begin{align} \label{RegularizedHamiltonian}
\frac12\Delta+
\left(
\frac{\pi^2}{8}
+
4\pi\beta\,\varepsilon
\right)
\frac1{\varepsilon^2}
v\!\left(\frac{\cdot}{\varepsilon}\right),
\qquad \varepsilon>0,    
\end{align}
where \(v\) is a bounded compactly supported potential satisfying $\int_{\mathbb R^3}v(x)\,dx=\frac{4\pi}{3}$. The limiting operator \(\Delta^\beta\) acts as the free Laplacian away from the origin and has domain
\[
\mathcal D(\Delta^\beta)
=
\Bigg\{
x\mapsto
\phi_k(x)
-
\frac{\phi_k(0)}
{\beta+\frac{ik}{4\pi}}
\frac{e^{ik|x|}}{4\pi|x|}
:
\begin{array}{l}
\phi_k\in H^{2,2}(\R^3),\
k\in\mathbb C,\
\operatorname{Im}k>0,\\
k^2\notin\{-16\pi^2\beta^2\}\cup[0,\infty)
\end{array}
\Bigg\}.
\]
The corresponding fundamental solution
$P_t^\beta : (\R^3\setminus\{0\})\times(\R^3\setminus\{0\})
\longrightarrow [0,\infty)$ of the point interaction heat equation \(\partial_tu=\frac12\Delta^\beta u\) is explicitly given by
\begin{align}\label{DefPointKer3dBeta} 
P_t^{\beta}(x,y)
\,:=\,
P_t(x,y)
+
\frac{2t}{|x||y|}\,P_t(|x|+|y|)
+
\frac{8\pi\beta\,t}{|x||y|}
\int_{0}^{\infty} e^{\,4\pi\beta u}\,
P_t\!\big(u+|x|+|y|\big)\,du,
\end{align}
for \(t>0\) and \(x,y\in\mathbb R^3\setminus\{0\}\), where $P_t(x,y) := (4\pi t )^{-3/2} e^{-|x-y|^{2}/(4t)}$ denotes the free heat kernel, and, with a slight abuse of notation, $P_t(r) := (4\pi t )^{-3/2} e^{-r^{2}/(4t)}$ for \(r\ge0\) denotes its radial version. The representation~\eqref{DefPointKer3dBeta} appears in several works on three-dimensional point interaction; see, for example,~\cite{Albeverio,Fleischmann,GrummtKolb,APT}.

Fix $T>0$. For $t \in [0,T]$, define the function $Q_t^\beta:\R^3 \to [0,\infty]$ by
\begin{align}\label{DefH}
Q_t^\beta(x)
\,:=\,
\int_{\R^3} P_t^\beta(x,y)\,dy-1
\,=\,
\frac{2t}{\beta\,|x|}
\int_{0}^{\infty}
\big(e^{4\pi\beta w}-1\big)\,
P_t(|x|+w)\,dw
\end{align}
for \(x\in\R^3\setminus\{0\}\), while \(Q_t^\beta(0):=\infty\). For \(0\le s<t\le T\), define the function $\mathlarger{p}^{\,T,\beta}_{s,t}:
(\R^3\setminus\{0\})\times(\R^3\setminus\{0\})
\to [0,\infty]$ by the Doob transform of the point interaction heat kernel \(P_{t-s}^\beta(x,y)\) as follows,
\begin{align}\label{FirstTrans}
\mathlarger{p}^{\,T,\beta}_{s,t}(x,y)
\,:=\,
\frac{1+Q^{\beta}_{T-t}(y)}{1+Q^{\beta}_{T-s}(x)}
\, P^{\beta}_{t-s}(x,y),
\qquad x,y\in\R^3\setminus\{0\}.
\end{align}
Let $\Omega:=C([0,T],\R^3)$ be the space of continuous paths equipped with the topology of uniform convergence and its Borel
\(\sigma\)-algebra \(\mathcal B(\Omega)\). Denote by
\(\{X_t\}_{t\in[0,T]}\) the coordinate process on \(\Omega\), that is,
\[
X_t(\omega):=\omega(t),
\qquad
\omega\in\Omega,\quad t\in[0,T],
\]
and let \(\{\mathcal F_t^X\}_{t\in[0,T]}\) be the filtration generated by \(X\). For every \(x\in\mathbb R^3\setminus\{0\}\), there exists a probability measure
\(\mathbf P_x^{T,\beta}\) on \((\Omega,\mathcal B(\Omega))\) under which
\(X_0=x\) almost surely and the coordinate process
\(\{X_t\}_{t\in[0,T]}\) is a time-inhomogeneous Markov process with transition
density \(\mathlarger p_{s,t}^{\,T,\beta}\) given by~\eqref{FirstTrans}; this follows from~\cite[Thm.~2.3]{CranstonKoralovMolchanovVainberg2}, where the parameters are related by \(\gamma=4\pi\beta\) and the time scaling \(t\mapsto t/2\). More precisely, for each \(\varepsilon>0\), consider the Gibbs measure \(\mathbf P_{x}^{T,\beta,\varepsilon}\) associated with the regularized
Schr\"odinger operators~\eqref{RegularizedHamiltonian} defined by the
Radon--Nikodym derivative
\[
\frac{d\mathbf P_{x}^{T,\beta,\varepsilon}}
     {d\mathbf P_x}(\omega)
=
\frac{1}{Z^{T,\beta,\varepsilon}(x)}
\exp\left\{
\left(
\frac{\pi^2}{8}
+
4\pi\beta\,\varepsilon
\right)
\int_0^T
\frac1{\varepsilon^2}
v\!\left(\frac{\omega(t)}{\varepsilon}\right)\,dt
\right\},
\]
where \(\mathbf P_x\) denotes Wiener measure started from \(x\), and
\(Z^{T,\beta,\varepsilon}(x)\) is the associated partition function. They proved that the measures $\bigl\{
\mathbf P_{x}^{T,\beta,\varepsilon}
\bigr\}_{\varepsilon>0}$ converge weakly to \(\mathbf P_x^{T,\beta}\) as \(\varepsilon\downarrow0\). See also~\cite{CranstonKoralovMolchanovVainberg1} for related work on Gibbs measures for the compactly supported potential case.

Since the transition density \(\mathlarger p_{s,t}^{\,T,\beta}(x,y)\)
in~\eqref{FirstTrans} is constructed from the point interaction heat kernel \(P_t^\beta(x,y)\), it is natural to expect that the corresponding diffusion \(\{X_t\}_{t\in[0,T]}\) under the law \(\mathbf P_x^{T,\beta}\) should interact non-trivially with the origin. However, this does not follow directly from the weak convergence construction of \(\mathbf P_x^{T,\beta}\), and Theorem~\ref{ThmSubMART}(i) confirms that the process does indeed visit the origin with positive probability. Another natural question concerns the stochastic dynamics of the diffusion. For fixed \(s\in[0,T]\) and \(x\in\mathbb R^3\setminus\{0\}\), the map $(t,y)\in[s,T]\times(\mathbb R^3\setminus\{0\}) \longmapsto \mathlarger p_{s,t}^{T,\beta}(x,y)$ satisfies the forward Kolmogorov equation away from the origin,
\begin{align}\label{KolmogorovForJ}
\frac{\partial}{\partial t}
\mathlarger p_{s,t}^{T,\beta}(x,y)
=
\frac12\Delta_y
\mathlarger p_{s,t}^{T,\beta}(x,y)
-
\nabla_y\cdot
\Big[
b_{T-t}^\beta(y)
\mathlarger p_{s,t}^{T,\beta}(x,y)
\Big],
\qquad y\neq0,
\end{align}
where the
drift is given by \(b_t^\beta(x):=\nabla_x\log\bigl(1+Q_t^\beta(x)\bigr)\). Thus, the process should satisfy a stochastic differential equation of the form
\begin{align}\label{SDEToSolve}
dX_t
=
dW_t
+
b_{T-t}^\beta(X_t)\,dt ,
\end{align}
with respect to a three-dimensional standard Brownian motion \(W\). However, the explicit formulas for \(P_t^\beta\) and \(Q_t^\beta\) show that, for fixed \(s\in(0,T)\), \(x\neq0\), and \(y\to0\),
\[
P_s^\beta(x,y)\asymp \frac1{|y|},
\qquad
1+Q_{T-s}^\beta(y)\asymp \frac1{|y|},
\qquad
|b_{T-s}^\beta(y)|\asymp \frac1{|y|}.
\]
Consequently, the absolute drift integral
\[
\mathbf E_x^{T,\beta}
\left[
\int_0^T
\big|b_{T-s}^{\beta}(X_s)\big|\,ds
\right]
=
\int_0^T
\int_{\R^3}
\mathlarger p_{0,s}^{\,T,\beta}(x,y)
\big|b_{T-s}^{\beta}(y)\big|\,dy\,ds
\]
may fail to be finite, since the integrand has a \(|y|^{-3}\)-type singularity as \(y\to0\). There is a substantial literature on stochastic differential equations with singular drifts; see, for instance,~\cite{Fukushima,Streit,Trutnau,Krylov,Zhang,Jin,Kinzebulatov}. However, none of the available general results appears to apply directly to the SDE~\eqref{SDEToSolve}. In the present paper, our goal is to study the visitation of the process to the origin, and we therefore establish the SDE only locally on excursion intervals away from the origin; see Proposition~\ref{PropositionLocalizedSDEAwayOrigin}. Nakashima~\cite{Nakashima} recently obtained an alternative probabilistic representation of the three-dimensional heat equation with a one-point interaction through a Feynman--Kac type formula; however, whether this local SDE description can be extended through visits to the origin remains an interesting open problem. In the corresponding two-dimensional model, the analogous SDE was solved globally in~\cite[Prop.~2.6]{CM}; see also~\cite{Chen1,Chen2,CM2,Mian} for related studies of the planar one-point interaction diffusion.

\section{Overview of the main results} \label{SectionMainResults}

In this section we give an overview of the main results of the paper.
Section~\ref{SubsectionStochdDiff} identifies a local stochastic differential equation satisfied by the diffusion \(\{X_t\}_{t\in[0,T]}\) away from the origin. Section~\ref{SubsectionSubMart} introduces a natural
submartingale associated with the process and studies its Doob-Meyer
decomposition. Finally, Section~\ref{SubsectionVisitation} uses the submartingale characterization to analyze the interaction of the diffusion with the origin, including its probability of visiting the origin.

\subsection{A local stochastic differential equation}
\label{SubsectionStochdDiff}

Fix \(T,\beta>0\). Given \(t\in[0,T]\), recall that the drift
\(b_t^\beta:\mathbb R^3\setminus\{0\}\to\mathbb R^3\) is given by
\begin{align}\label{DefDriftFun}
b_t^\beta(x)
:=
\nabla\log\bigl(1+Q_t^\beta(x)\bigr)
=
\frac{\nabla Q_t^\beta(x)}
     {1+Q_t^\beta(x)}, \qquad x\neq0,
\end{align}
where \(\nabla=(\partial_{x_1},\partial_{x_2},\partial_{x_3})\) denotes the gradient operator. Since \(Q_t^\beta(0)=\infty\), the
drift is not defined at the origin. For \(x\in\mathbb R^3\setminus\{0\}\), let \(\mathcal N_x^T\) denote the collection of all subsets of \(\mathbf P_x^{T,\beta}\)-null sets and define
\begin{align}\label{Augmented}
\mathcal F_t^{T,x}
:=
\sigma\big\{
\mathcal F_t^X\cup\mathcal N_x^T
\big\},
\qquad
\mathcal B_x^T
:=
\sigma\big\{
\mathcal B(\Omega)\cup\mathcal N_x^T
\big\}.
\end{align}
The measure \(\mathbf P_x^{T,\beta}\) extends uniquely to the augmented \(\sigma\)-algebra \(\mathcal B_x^T\), and we use the same symbol for the extension. The augmented filtration $\mathcal F^{T,x} := \{\mathcal F_t^{T,x}\}_{t \in [0, T]}$ will be used throughout this work. The proof of the following proposition is in Section~\ref{SectionWeakSolConst}.

\begin{proposition}\label{PropositionLocalizedSDEAwayOrigin}
Fix \(T,\beta>0\) and \(x\in\mathbb R^3\setminus\{0\}\). There exists a three-dimensional standard Brownian motion $\{W_t^{T,\beta}\}_{t \in [0,T]}$ on the probability space
\((\Omega,\mathcal B_x^T,\mathbf P_x^{T,\beta})\)
with respect to the filtration $\{\mathcal F_t^{T,x}\}_{t \in [0, T]}$ such that on every interval \([a,b]\subset[0,T]\) such that \(X_s\neq0\) for all \(s\in[a,b]\), the process \(X\) satisfies the stochastic differential equation~\eqref{SDEToSolve} with \(W\) replaced by \(W^{T,\beta}\).
\end{proposition}

\subsection{A submartingale characterization}
\label{SubsectionSubMart}

For $t \in [0,T]$, define the function $S_t^{\beta}:\R^2\rightarrow [0,1]  $ with
\begin{align}\label{DefpFunction}
S_t^{\beta}(x):=\frac{1}{1+Q_t^{\beta}( x)} .
\end{align}
The function \(S_t^\beta\) satisfies the following partial differential equation for \(t>0\)
\begin{align}\label{PartialForP}
\frac{\partial}{\partial t}\,S_t^{\beta}(x)\,=\,  \frac{1}{2}\, \Delta \,S_t^{\beta}(x)\,+\, b_t^{\beta}(x)\cdot (\nabla\,S_t^{\beta})(x)\,, \hspace{.5cm} x\in \R^3\backslash \{0\}\,,
\end{align}
because $\frac{\partial}{\partial t} Q_t^{\beta}( x)= \frac{1}{2} \Delta Q_t^{\beta}( x)$. Consider the process $\{\mathbf S^{T,\beta}_t\}_{t\in [0,T]}$ given by $ \mathbf S^{T,\beta}_t:=S_{T-t}^{\beta}(X_t)$. The proof of the following proposition is in Section~\ref{ProofSubMartCharacterization}.

\begin{proposition}\label{PropSubMart}
Fix \(T,\beta>0\) and \(x\in\mathbb R^3\setminus\{0\}\). The process
\(\{\mathbf S_t^{T,\beta}\}_{t\in[0,T]}\) is a bounded, continuous
\(\mathbf P_x^{T,\beta}\)-submartingale with respect to
\(\{\mathcal F_t^{T,x}\}_{t\in[0,T]}\). Moreover, if
\(\mathbf S^{T,\beta}=\mathbf M^{T,\beta}+\mathbf A^{T,\beta}\) denotes its
Doob--Meyer decomposition, then the increasing component
\(\mathbf A^{T,\beta}\) is constant during the excursions of \(X\) away from
the origin, and the martingale component satisfies
\(\mathbf M^{T,\beta}\in L^2(\mathbf P_x^{T,\beta})\) and is given by
\begin{align}\label{MartingaleFormula}
\mathbf M_t^{T,\beta}
=
S_T^\beta(x)
-
\int_0^t
\mathbf S_s^{T,\beta}
b_{T-s}^{\beta}(X_s)\cdot dW_s^{T,\beta},
\qquad t\in[0,T].
\end{align}
\end{proposition}

\begin{remark}\label{RemarkSDES}
Since the increasing component \(\mathbf A^{T,\beta}\) in
Proposition~\ref{PropSubMart} is continuous, remains constant during excursions
of \(X\) away from the origin, and satisfies
\(\mathbf A_0^{T,\beta}=0\), it follows that
\(\mathbf A_t^{T,\beta}=0\) for all \(0\le t<\tau\), where $\tau:=\inf\{t\in[0,T]:X_t=0\}$. Consequently,
\begin{align}\label{SDEforSuptoTau}
d\mathbf S_t^{T,\beta}
=
d\mathbf M_t^{T,\beta}
\overset{\eqref{MartingaleFormula}}{=}
-
\mathbf S_t^{T,\beta}
b_{T-t}^{\beta}(X_t)\cdot dW_t^{T,\beta},
\qquad 0\le t<\tau.
\end{align}
\end{remark}

\subsection{Visitation of the origin and conditional laws}
\label{SubsectionVisitation}

Recall that
\[
\tau:=\inf\{t\in[0,T]:X_t=0\}
\]
denotes the first time at which the diffusion \(\{X_t\}_{t\in[0,T]}\)
visits the origin. The proof of following theorem is in Section~\ref{SubsectionThmSubMART}.

\begin{theorem}\label{ThmSubMART}
Fix \(T,\beta>0\) and \(x\in\mathbb R^3\setminus\{0\}\). Then the following hold.

\begin{enumerate}[(i)]

\item $\displaystyle \mathbf P_x^{T,\beta}[\tau\le T]
=
\frac{Q_T^\beta(x)}
{1+Q_T^\beta(x)}$

\item
Under the conditional law \(\mathbf P_x^{T,\beta}[\;\cdot\mid\tau>T]\), the process
\(\{X_t\}_{t\in[0,T]}\) is a three-dimensional Brownian motion started from \(x\).

\item Let \(\sigma\) be an \(\mathcal F^X\)-stopping time such that
\(\mathbf P_x^{T,\beta}(\sigma<\tau)>0\). Then, for every
\(A\in\mathcal F_{\sigma}^{T,x}\),
\[
\mathbf P_x^{T,\beta}
\big[
A\,\big|\,\sigma<\tau
\big]
=
\frac{
\mathbf E_x
\left[
\mathbf 1_A
\big(
1+Q_{T-\sigma}^{\beta}(X_{\sigma})
\big)
\right]
}{
\mathbf E_x
\left[
1+Q_{T-\sigma}^{\beta}(X_{\sigma})
\right]
},
\]
where \(\mathbf E_x\) denotes expectation with respect to the Wiener measure \(\mathbf P_x\) on \(\Omega=C([0,T],\mathbb R^3)\), started from \(x\).

\item
Conditioned on the event \(\{\tau\le T\}\), the hitting time \(\tau\) has
survival function
\[
\mathbf P_x^{T,\beta}
\big[
\tau>s
\,\big|\,
\tau\le T
\big]
=
\frac{
(P_s*Q_{T-s}^\beta)(x)
}{
Q_T^\beta(x)
},
\qquad 0\le s\le T,
\]
where \(*\) denotes convolution on \(\mathbb R^3\) and \(P_s\) is the
three-dimensional free heat kernel.
\end{enumerate}
\end{theorem}

\subsection{Organization of the paper}

The remainder of the paper is organized as follows.
\begin{itemize} 
\item[--]
Section~\ref{SectionWeakSolConst} is devoted to the construction of a weak
solution to the localized stochastic differential equation
\eqref{SDEToSolve}. We first define a space--time harmonic function
\(\mathbb Q_t^\beta\) in Section~\ref{TheHarmonicFunction}, then construct the Brownian motion \(W^{T,\beta}\) in Section~\ref{BrownianMotionConstruction}, and finally prove
Proposition~\ref{PropositionLocalizedSDEAwayOrigin} in Section~\ref{ProofPropositionLocalizedSDEAwayOrigin}.

\item[--] Section~\ref{ProofSubMartCharacterization} establishes the submartingale characterization of the diffusion through the proof of
Proposition~\ref{PropSubMart}.

\item[--] Section~\ref{SubsectionThmSubMART} is devoted to the proof of Theorem~\ref{ThmSubMART}.

\item[--]
Section~\ref{ProofAuxLemmas} contains the proofs of the auxiliary results.
\end{itemize}

\section{A weak solution to the SDE~(\ref{SDEToSolve}) away from the origin}\label{SectionWeakSolConst}

In this section we construct a weak solution to SDE~\eqref{SDEToSolve} away from the origin. We use a space--time harmonic transformation \(\mathbb Q^\beta\) and an associated \(\mathbf P_x^{T,\beta}\)-martingale \(\{Y_t^{T,\beta}\}_{t\in[0,T]}\), defined by \(Y_t^{T,\beta}:=\mathbb Q_{T-t}^{\beta}(X_t)\), to obtain the three-dimensional Brownian motion \(W^{T,\beta}\) in Proposition~\ref{PropositionLocalizedSDEAwayOrigin}.

\subsection{The space-time harmonic function \texorpdfstring{$\boldsymbol{\mathbb Q_t^\beta(x)}$}{Lg}} \label{TheHarmonicFunction}

Consider the function \(\mathbb Q_t^\beta:\mathbb R^3\setminus\{0\}\to\mathbb R^3\setminus\{0\}\) defined by 
\begin{align}\label{DefUpsilon}
\mathbb Q_t^{\beta}(x)\,:=\frac{x}{1+Q_t^{\beta}(x)  }\,,
\end{align}
where $Q^{\beta}_t(x)$ is defined as in~(\ref{DefH}) for $t > 0$ and $Q^{\beta}_t(x):=0$ for $t = 0$. The proof of the following lemma is in Section~\ref{ProofLemmaL2IncrementY}.

\begin{lemma}\label{LemmaL2IncrementY}
Fix \(T,\beta>0\) and \(x\in\mathbb R^3\setminus\{0\}\). Then, for every
\(t \in [0, T]\),
\[
\int_{\mathbb R^3}
\mathlarger p_{0,t}^{T,\beta}(x,y)
\left|
\mathbb Q_{T-t}^{\beta}(y)-\mathbb Q_t^\beta(x)
\right|^2
\,dy
<\infty .
\]
\end{lemma}

Given \(\beta>0\) and \(T\ge0\), the radial symmetry of
\(Q_T^\beta(x)\) implies that the Jacobian matrix
\(\nabla^\dagger\mathbb Q_T^\beta(x)\) has the form
\[
\nabla^\dagger\mathbb Q_T^\beta(x)
=
\lambda_T^\beta(x)\,P_x
+
\mu_T^\beta(x)\,(I-P_x),
\]
where \(P_x:\R^3\to\R^3\) denotes the orthogonal projection onto
\(\mathrm{span}\{x\}\), and the corresponding eigenvalues are
\begin{align}\label{SigmaForm}
\lambda_{T}^{\beta}(x)\,:=\,\frac{1+Q_{T}^{\beta}(x)\,+\,|x|\,\big|\nabla Q_{T}^{\beta}(x)\big|}{\big(1+Q_{T}^{\beta}(x)\big)^2  } \hspace{.7cm}\text{and}\hspace{.7cm} \mu_{T}^{\beta}(x)\,:=\,\frac{1}{1+Q_{T}^{\beta}(x)  }\,.
\end{align} 
Since \(\lambda_t^\beta(x)>0\) and
\(\mu_t^\beta(x)>0\), the \(3\times 3\) matrix
\(\nabla^\dagger\mathbb Q_t^\beta(x)\) is invertible for every
\(x\in\R^3\setminus\{0\}\) and
\begin{align}\label{SigmaInverse}
\bigl(\nabla^\dagger\mathbb Q_t^\beta(x)\bigr)^{-1}
=
\frac{1}{\lambda_t^\beta(x)}\,P_x
+
\frac{1}{\mu_t^\beta(x)}\,(I-P_x). 
\end{align} 
Given a twice-differentiable function \(f:[0,T]\times\mathbb R^3\to\mathbb R\)
with compact support in \(\mathbb R^3\setminus\{0\}\), the Kolmogorov
equation~\eqref{KolmogorovForJ} implies that
\begin{align}\label{KolmogorovForJBack}
\frac{\partial}{\partial h}
\int_{\mathbb R^3}
\mathlarger p_{t,t+h}^{T,\beta}(x,y)\,
f(t+h,y)\,dy\,\bigg|_{h=0}
=
\Big(
\mathcal L_x^{T-t,\beta}
+
\frac{\partial}{\partial t}
\Big)f(t,x),
\end{align}
where
\[
\mathcal L_x^{r,\beta}
:=
\frac12\Delta_x
+
\frac{\nabla_x Q_r^\beta(x)}
     {1+Q_r^\beta(x)}
\cdot\nabla_x,
\qquad r>0,\quad x\in\mathbb R^3\setminus\{0\}.
\]
Note that we apply
differential operators to vector-valued functions component-wise; for example,
if \(F:\mathbb R^3\to\mathbb R^3\) and \(F=(F_1,F_2,F_3)\), then
\[
\Delta F:=(\Delta F_1,\Delta F_2,\Delta F_3).
\]
The next lemma records the basic properties of \(\mathbb Q^\beta\); its proof is omitted.
\begin{lemma}\label{LemmaMartFUN}
Let \(T,\beta>0\). The following statements hold.

\begin{enumerate}
\item[(i)]
For \(x\in\mathbb R^3\setminus\{0\}\) and
\(0\le s<t<\infty\),
\begin{align}\label{MartFunEQ}
\int_{\R^3}
\mathlarger p_{s,t}^{T,\beta}(x,y)\,
\mathbb Q_{T-t}^{\beta}(y)\,dy
=
\mathbb Q_{T-s}^{\beta}(x).
\end{align}

\item[(ii)]
The \(\mathbb R^3\)-valued function
\((t,x)\mapsto \mathbb Q_t^\beta(x)\) satisfies
\[
\frac{\partial}{\partial t}\mathbb Q_t^\beta(x)
=
\mathcal L_x^{t,\beta}\mathbb Q_t^\beta(x),
\qquad
x\in\mathbb R^3\setminus\{0\}.
\]

\item[(iii)]
For any twice-differentiable function \(G:\mathbb R^3\to\mathbb R\),
\begin{align}\label{ChainRULE}
\Big(
\mathcal L_x^{t,\beta}
-
\frac{\partial}{\partial t}
\Big)
G\big(\mathbb Q_t^\beta(x)\big)
=
\frac12
\bigg[
\big(\lambda_t^\beta(x)\big)^2
\partial_{x\parallel}^2G(z)
+
\big(\mu_t^\beta(x)\big)^2
\sum_{j=1}^{2}\partial_{e_j}^2G(z)
\bigg]_{z=\mathbb Q_t^\beta(x)},
\end{align}
where \(\{e_1,e_2\}\) is any orthonormal basis of the plane $x^\perp := \{v\in\mathbb R^3:v\cdot x=0\}$, and
\[
\partial_{x\parallel}^2G(z)
:=
\left(\frac{x}{|x|}\cdot\nabla\right)^2G(z),
\qquad
\partial_{e_j}^2G(z)
:=
(e_j\cdot\nabla)^2G(z),
\quad j=1,2.
\]
\end{enumerate}
\end{lemma}

\subsection{Construction of the Brownian motion \texorpdfstring{$\boldsymbol{W^{T,\beta}}$}{Lg}}\label{BrownianMotionConstruction}

For \(x\in\mathbb R^3\setminus\{0\}\), recall that the augmented
\(\sigma\)-algebra \(\mathcal B_x^{T}\) and the augmented filtration
\(\{\mathcal F_t^{T,x}\}_{t \in [0, T]}\) are defined as in~\eqref{Augmented}.
The next lemma and its corollary concern the construction of the
three-dimensional Brownian motion \(W^{T,\beta}\) in
Proposition~\ref{PropositionLocalizedSDEAwayOrigin}. Recall that \(\mathbb Q_t^\beta\) is defined by~\eqref{DefUpsilon} on
\((0,T]\times(\mathbb R^3\setminus\{0\})\), and since
\(\lim_{x\to0}\mathbb Q_t^\beta(x)=0\) for every \(t>0\), it extends
continuously to \((0,T]\times\mathbb R^3\) by setting
\(\mathbb Q_t^\beta(0):=0\) for \(t>0\). In what follows,
\(\mathbb Q_t^\beta\) always denotes this continuous extension.

\begin{lemma}\label{LemmaMartI}
Fix \(T,\beta>0\) and \(x\in\mathbb R^3\setminus\{0\}\). Define the \(\mathbb R^3\)-valued process
\[
Y_t^{T,\beta}
=
\mathbb Q_{T-t}^{\beta}(X_t),
\qquad 0\le t\le T.
\]

\begin{enumerate}
\item[(i)]
The process \(Y^{T,\beta}\) is a continuous square-integrable
\(\mathbf P_x^{T,\beta}\)-martingale with respect to the augmented filtration
\(\{\mathcal F_t^{T,x}\}_{t\in[0,T]}\).

\item[(ii)]
For every \(u\in\mathbb R^3\), the quadratic variation of the real-valued
martingale \(\{u\cdot Y_t^{T,\beta}\}_{t\in[0,T]}\) is given by
\[
\big\langle u\cdot Y^{T,\beta}\big\rangle_t
=
\int_0^t
\big\|\nabla^\dagger\mathbb Q_{T-s}^{\beta}(X_s)u\big\|_2^2\,ds,
\qquad 0\le t\le T.
\]
\end{enumerate}
\end{lemma}

\begin{proof} \noindent Part (i). Since \(X\) is continuous under \(\mathbf P_x^{T,\beta}\) and
\((r,z)\mapsto\mathbb Q_r^\beta(z)\) is continuous on
\((0,T]\times\mathbb R^3\), it follows that
\(Y^{T,\beta}\) is continuous \(\mathbf P_x^{T,\beta}\)-almost surely.

Given \(t \in [0, T]\), using Minkowski's inequality we obtain the first inequality below.
\begin{align*}
\mathbf E_x^{T,\beta}
\Big[
|Y_t^{T,\beta}|^2
\Big]^{1/2}
\le\,&
\mathbf E_x^{T,\beta}
\Big[
|Y_0^{T,\beta}|^2
\Big]^{1/2}
+
\mathbf E_x^{T,\beta}
\Big[
|Y_t^{T,\beta}-Y_0^{T,\beta}|^2
\Big]^{1/2}
\\
=\,&
\big|\mathbb Q_T^\beta(x)\big|
+
\left(
\int_{\R^3}
\big|
\mathbb Q_{T-t}^\beta(y)
-
\mathbb Q_T^\beta(x)
\big|^2
\mathlarger p_{0,t}^{T,\beta}(x,y)
\,dy
\right)^{1/2}
< \infty
\end{align*}
The final inequality follows from the bound $|\mathbb Q_T^\beta(x)|\le |x|$ and Lemma~\ref{LemmaL2IncrementY}. Thus \(Y^{T,\beta}\) is square-integrable.

For fixed \(0\le s<t\le T\), by the Markov property of \(\mathbf P_x^{T,\beta}\) with transition density \(\mathlarger p_{s,t}^{T,\beta}\), we have the first equality below.
\[
\mathbf E_x^{T,\beta}
\left[
Y_t^{T,\beta}
\,\middle|\,
\mathcal F_s^X
\right]
=
\int_{\R^3}
\mathlarger p_{s,t}^{T,\beta}(X_s,y)
\mathbb Q_{T-t}^{\beta}(y)\,dy
=
\mathbb Q_{T-s}^{\beta}(X_s)
=
Y_s^{T,\beta}
\]
The second equality uses~\eqref{MartFunEQ}. Thus \(Y^{T,\beta}\) is a martingale with respect to \(\{\mathcal F_t^X\}_{t \in [0,T]}\). Since \(\mathcal F_t^{T,x}\) is the augmentation of
\(\mathcal F_t^X\) by \(\mathbf P_x^{T,\beta}\)-null sets, as defined in~\eqref{Augmented}, the martingale property extends to
\(\{\mathcal F_t^{T,x}\}_{t\in[0,T]}\). \vspace{.2cm}

\noindent Part (ii). For \(u\in\R^3\) and \(0\le t<T\), the density of the quadratic variation of the continuous martingale \(u\cdot Y^{T,\beta}\) is, \(\mathbf P_x^{T,\beta}\)-almost surely, given by
\begin{align}
\frac{d}{dt}
\big\langle u\cdot Y^{T,\beta}\big\rangle_t
&=
\lim_{h\downarrow0}
\frac1h\,
\mathbf E_x^{T,\beta}
\left[
\left(
u\cdot Y_{t+h}^{T,\beta}
-
u\cdot Y_t^{T,\beta}
\right)^2
\,\middle|\,
\mathcal F_t^{T,x}
\right] \nonumber \\
&=
\lim_{h\downarrow0}
\frac1h\,
\mathbf E_{X_t}^{T-t,\beta}
\left[
\left(
u\cdot Y_h^{T-t,\beta}
-
u\cdot Y_0^{T-t,\beta}
\right)^2
\right] \nonumber \\
&=
\lim_{h\downarrow0}
\frac1h
\int_{\R^3}
\left(
u\cdot\mathbb Q_{T-t-h}^{\beta}(y)
-
u\cdot\mathbb Q_{T-t}^{\beta}(X_t)
\right)^2
\mathlarger p_{0,h}^{T-t,\beta}(X_t,y)\,dy , \nonumber
\end{align}
where the second equality follows from the Markov property of
\(\mathbf P_x^{T,\beta}\), applied at time \(t\), and the final equality uses that, under \(\mathbf P_{X_t}^{T-t,\beta}\), the random variable \(X_h\) has density \(\mathlarger p_{0,h}^{T-t,\beta}(X_t,\cdot)\). Equivalently, writing the right-hand side above in derivative notation, we obtain the first equality below.
\begin{align}  \label{EqBracketDensity}
\frac{d}{dt}\big\langle u\cdot Y^{T,\beta}\big\rangle_t
=\,&
\frac{\partial}{\partial h}
\int_{\R^3}
\left(
u\cdot\mathbb Q_{T-t-h}^{\beta}(y)
-
u\cdot\mathbb Q_{T-t}^{\beta}(X_t)
\right)^2
\mathlarger p_{0,h}^{T-t,\beta}(X_t,y)\,dy
\bigg|_{h=0} \nonumber \\
=\,&
\bigg(
\mathcal L_y^{T-t,\beta}
+
\frac{\partial}{\partial s}
\bigg)
\Big(
u\cdot \mathbb Q_{T-s}^{\beta}(y)
-
u\cdot \mathbb Q_{T-t}^{\beta}(X_t)
\Big)^2
\bigg|_{(s,y)=(t,X_t)} \nonumber \\
=\,&
\big(\lambda_{T-t}^{\beta}(X_t)\big)^2
\left(u\cdot \frac{X_t}{|X_t|}\right)^2
+
\big(\mu_{T-t}^{\beta}(X_t)\big)^2
\sum_{j=1}^{2}(u\cdot e_j)^2  
\end{align}
The second equality applies~\eqref{KolmogorovForJBack}, and the final equality follows from the three-dimensional chain-rule identity~\eqref{ChainRULE}, where \(\{e_1,e_2\}\) is an orthonormal basis of the plane $X_t^\perp = \{v\in\mathbb R^3:v\cdot X_t=0\}$. Finally, using the orthogonal decomposition $u
=
P_{X_t}u
+
(I-P_{X_t})u$ together with
\[
\|P_{X_t}u\|_2^2
=
\left(u\cdot \frac{X_t}{|X_t|}\right)^2
\quad \textup{and} \quad
\|(I-P_{X_t})u\|_2^2
=
\sum_{j=1}^2 (u\cdot e_j)^2 ,
\]
the right-hand side of~\eqref{EqBracketDensity} may be written as in the first equality below.
\begin{align}
\frac{d}{dt}\big\langle u\cdot Y^{T,\beta}\big\rangle_t
=\,
\big(\lambda_{T-t}^{\beta}(X_t)\big)^2
\left\|P_{X_t}u\right\|_2^2
+
\big(\mu_{T-t}^{\beta}(X_t)\big)^2
\left\|(I-P_{X_t})u\right\|_2^2
=
\left\|
\nabla^\dagger\mathbb Q_{T-t}^{\beta}(X_t)u
\right\|_2^2 \nonumber
\end{align}
Integrating the above identity over \([0,t]\) yields the desired formula.
\end{proof}

\begin{remark}\label{RemarkOriginNull}
Since the process \(X\) under \(\mathbf P_x^{T,\beta}\) admits
transition densities \(\mathlarger p_{s,t}^{T,\beta}(x,y)\), the random set
\[
\{t\in[0,T]:X_t=0\}
\]
has Lebesgue measure zero, \(\mathbf P_x^{T,\beta}\)-almost surely. Indeed,
\begin{align*}
\mathbf E_x^{T,\beta}
\Big[
\operatorname{meas}
\big(
\{s\in[0,T]:X_s=0\}
\big)
\Big]
&=
\mathbf E_x^{T,\beta}
\bigg[
\int_0^T
\mathbf 1_{\{X_s=0\}}
\,ds
\bigg]  \nonumber \\
&=
\int_0^T
\int_{\mathbb R^3}
\mathlarger p_{0,s}^{T,\beta}(x,y)
\mathbf 1_{\{y=0\}}
\,dy\,ds 
=
0,
\end{align*}
where \(\operatorname{meas}(E)\) denotes the Lebesgue measure of a measurable set \(E\subset[0,T]\). Consequently,
\[
\int_0^t
\mathbf 1_{\{X_s\neq0\}}
\,ds
=
t,
\qquad
0\le t\le T,
\]
\(\mathbf P_x^{T,\beta}\)-almost surely. In particular, any indicator
\(\mathbf 1_{\{X_s\neq0\}}\) appearing inside a time integral may be omitted
without affecting the value of the integral.
\end{remark}

Recall that the \(3\times 3\) matrix \(\nabla^\dagger\mathbb Q_T^\beta(x)\) is invertible for \(x\neq0\), with inverse given by~\eqref{SigmaInverse}.

\begin{corollary}\label{CorollaryMartI}
Fix \(T,\beta>0\) and \(x\in\mathbb R^3\setminus\{0\}\). Let the
\(\mathbb R^3\)-valued process
\(\{Y_t^{T,\beta}\}_{t \in [0, T]}\) be defined as in
Lemma~\ref{LemmaMartI}. The process
\begin{align}\label{DefSBMW}
W_t^{T,\beta}
:=
\int_0^t
\mathbf 1_{\{X_s\neq0\}}
\big(\nabla^\dagger\mathbb Q_{T-s}^{\beta}(X_s)\big)^{-1}
\,dY_s^{T,\beta},
\qquad t \in [0, T],
\end{align}
is a three-dimensional standard Brownian motion with respect to
\(\{\mathcal F_t^{T,x}\}_{t \in [0, T]}\) under \(\mathbf P_x^{T,\beta}\).
Moreover,
\[
dY_t^{T,\beta}
=
\nabla^\dagger\mathbb Q_{T-t}^{\beta}(X_t)\,dW_t^{T,\beta}.
\]
\end{corollary}

\begin{proof}
By Lemma~\ref{LemmaMartI}(ii), the matrix-valued quadratic variation of
\(Y^{T,\beta}\) is
\begin{align} \label{QuadVarY}
d\langle Y^{T,\beta}\rangle_s
=
\nabla^\dagger\mathbb Q_{T-s}^{\beta}(X_s)
\big(\nabla^\dagger\mathbb Q_{T-s}^{\beta}(X_s)\big)^{\dagger}\,ds .
\end{align}
Since \(\nabla^\dagger\mathbb Q_{T-s}^{\beta}(X_s)\) is invertible on
\(\{X_s\neq0\}\), the stochastic integral defining
\(W^{T,\beta}\) is well defined and hence \(W^{T,\beta}\) is a
continuous local martingale with respect to
\(\{\mathcal F_t^{T,x}\}_{t\in[0,T]}\).

Next, to compute the quadratic variation of \(W^{T,\beta}\), we use~\eqref{DefSBMW} to obtain the first equality below.
\begin{align*}
d\langle W^{T,\beta}\rangle_s
&=
\mathbf 1_{\{X_s\neq0\}}
\big(\nabla^\dagger\mathbb Q_{T-s}^{\beta}(X_s)\big)^{-1}
\,d\langle Y^{T,\beta}\rangle_s\,
\big(\big(\nabla^\dagger\mathbb Q_{T-s}^{\beta}(X_s)\big)^{-1}\big)^\dagger  \\
&=
\mathbf 1_{\{X_s\neq0\}}
\big(\nabla^\dagger\mathbb Q_{T-s}^{\beta}(X_s)\big)^{-1}
\nabla^\dagger\mathbb Q_{T-s}^{\beta}(X_s)
\big(\nabla^\dagger\mathbb Q_{T-s}^{\beta}(X_s)\big)^\dagger
\big(\big(\nabla^\dagger\mathbb Q_{T-s}^{\beta}(X_s)\big)^{-1}\big)^\dagger \,ds  \\
&=
\mathbf 1_{\{X_s\neq0\}} I_3\,ds 
\end{align*}
The second equality uses~\eqref{QuadVarY}. Therefore, by Remark~\ref{RemarkOriginNull}, for every \(t\in[0,T]\),
\[
\langle W^{T,\beta}\rangle_t
=
I_3\int_0^t \mathbf 1_{\{X_s\neq0\}}\,ds
=
tI_3,
\qquad
\mathbf P_x^{T,\beta}\text{-a.s.}
\]
Since \(W^{T,\beta}\) is a continuous local martingale starting from \(0\) with quadratic variation \(\langle W^{T,\beta}\rangle_t=tI_3\), L\'evy's characterization theorem implies that \(W^{T,\beta}\) is a three-dimensional standard Brownian motion with respect to \(\{\mathcal F_t^{T,x}\}_{t\in[0,T]}\) under \(\mathbf P_x^{T,\beta}\).

Finally, from the definition~\eqref{DefSBMW} of \(W^{T,\beta}\), we have
\[
dW_t^{T,\beta}
=
\mathbf 1_{\{X_t\neq0\}}
\big(\nabla^\dagger\mathbb Q_{T-t}^{\beta}(X_t)\big)^{-1}
\,dY_t^{T,\beta}.
\]
Multiplying both sides by \(\nabla^\dagger\mathbb Q_{T-t}^{\beta}(X_t)\) and using
Remark~\ref{RemarkOriginNull}, we obtain the desired identity.
\end{proof}

\subsection{Proof of Proposition~\ref{PropositionLocalizedSDEAwayOrigin}}
\label{ProofPropositionLocalizedSDEAwayOrigin}

Before proceeding, we record the derivative conventions used throughout this section. Let \(\Phi:\mathbb R^3\to\mathbb R^3\) be a twice continuously differentiable map.
\begin{enumerate}[--]
\item
The Jacobian matrix of \(\Phi\) at \(x\in\mathbb R^3\) is denoted by
\(\nabla^\dagger\Phi(x)\).

\item
The second derivative \(\Hess\Phi(x)\) is viewed as an element of
\(\mathbb R^3\otimes M^{3\times3}\) such that
\[
(u^\dagger\otimes I_3)\,\Hess\Phi(x)
\]
coincides with the Hessian matrix of the scalar-valued function
\(u^\dagger\Phi=u\cdot\Phi\) for every \(u\in\mathbb R^3\).

\item
If \(\Psi\) denotes the inverse of \(\Phi\) and
\(\nabla^\dagger\Phi(x)\) is invertible, then differentiating the
identity $\Psi(\Phi(x))=x$ twice and applying the chain rule yields
\begin{align}\label{Pip}
(\Hess\Psi)\bigl(\Phi(x)\bigr)
&=
-
\Big(
\nabla^\dagger\Phi(x)
\otimes
\nabla\Phi(x)
\Big)^{-1}
(\Hess\Phi)(x)
\Big(
I_3\otimes
\nabla^\dagger\Phi(x)
\Big)^{-1},
\end{align}
where \(\nabla\Phi(x)\) denotes the transpose of
\(\nabla^\dagger\Phi(x)\), and we use the identity $(A\otimes B)^{-1} = A^{-1}\otimes B^{-1}$ for invertible matrices \(A,B\in M^{3\times3}\).
\end{enumerate}

Recall that, for \(t\ge0\), the function
\(\mathbb Q_t^\beta:\mathbb R^3\setminus\{0\}\to\mathbb R^3\setminus\{0\}\)
is defined in~\eqref{DefUpsilon}. Since
\(\nabla^\dagger\mathbb Q_t^\beta(x)\)
is invertible for every \(x\neq0\), the map \(\mathbb Q_t^\beta\) is a
\(C^2\)-diffeomorphism of \(\mathbb R^3\setminus\{0\}\). We denote its inverse by
\[
\overleftarrow{\mathbb Q}_t^\beta:\mathbb R^3\setminus\{0\}\to\mathbb R^3\setminus\{0\}.
\]
The next lemma records two identities for the functions \(\mathbb Q_t^\beta\) and \(\overleftarrow{\mathbb Q}_t^\beta\) that will be used in the proof of Proposition~\ref{PropositionLocalizedSDEAwayOrigin}.

\begin{lemma}\label{LemmaInversePhiIdentities}
For every \(t>0\) and \(x\in\mathbb R^3\setminus\{0\}\),
\begin{align}
&\big(\nabla^\dagger \overleftarrow{\mathbb Q}_t^\beta\big)\big(\mathbb Q_t^\beta(x)\big)
\,\nabla^\dagger\mathbb Q_t^\beta(x)
=
I_3 , \label{PDEUp1}
\\
&
\Big(\frac{\partial}{\partial t}\overleftarrow{\mathbb Q}_t^\beta\Big)
\big(\mathbb Q_t^\beta(x)\big)
=
- b_t^\beta(x)
+
\frac12
\Tr_2
\Big[
\big(I_3\otimes (\nabla^\dagger\mathbb Q_t^\beta(x))^\dagger\big)
(\Hess \overleftarrow{\mathbb Q}_t^\beta)
\big(\mathbb Q_t^\beta(x)\big)
\big(I_3\otimes \nabla^\dagger\mathbb Q_t^\beta(x)\big)
\Big]. \label{PDEUp2}
\end{align}
\end{lemma}

\begin{proof}
Since \(\overleftarrow{\mathbb Q}_t^\beta\) is the inverse of \(\mathbb Q_t^\beta\) on
\(\mathbb R^3\setminus\{0\}\), we have
\[
\overleftarrow{\mathbb Q}_t^\beta(\mathbb Q_t^\beta(x))=x,
\qquad x\in\mathbb R^3\setminus\{0\}.
\]
Differentiating with respect to \(x\) and using the chain rule, we obtain~\eqref{PDEUp1}.

To prove~\eqref{PDEUp2}, differentiating the identity \(\overleftarrow{\mathbb Q}_t^\beta(\mathbb Q_t^\beta(x))=x\) with respect to \(t\) and applying the chain rule yields the first equality below.
\begin{align}\label{I1}
\Big(\frac{\partial}{\partial t}\overleftarrow{\mathbb Q}_t^\beta\Big)
\big(\mathbb Q_t^\beta(x)\big)
=\,&
-
\big(\nabla^\dagger\overleftarrow{\mathbb Q}_t^\beta\big)
\big(\mathbb Q_t^\beta(x)\big)
\Big(\frac{\partial}{\partial t}\mathbb Q_t^\beta\Big)(x) \nonumber\\
=\,&
\big(\nabla^\dagger\mathbb Q_t^\beta(x)\big)^{-1}
\frac{
\big(\frac{\partial}{\partial t}Q_t^\beta\big)(x)
}{
\big(1+Q_t^\beta(x)\big)^2
}
x 
=
\frac{1}{2\,\lambda_t^\beta(x)}
\frac{
\Delta Q_t^\beta(x)
}{
\big(1+Q_t^\beta(x)\big)^2
}
x 
\end{align}
The second equality follows from~\eqref{PDEUp1} and a direct differentiation of \(\mathbb Q_t^\beta\) in~\eqref{DefUpsilon} with respect to \(t\), while the final equality uses the heat equation \(\frac{\partial}{\partial t}Q_t^\beta(x)=\frac12\Delta Q_t^\beta(x)\) for \(x\in\mathbb R^3\setminus\{0\}\) together with the fact that \(x\) is an eigenvector of \(\big(\nabla^\dagger\mathbb Q_t^\beta(x)\big)^{-1}\) with eigenvalue \(\lambda_t^\beta(x)^{-1}\).

Next, using the definition of \(\mathbb Q_t^\beta\) in~\eqref{DefUpsilon}, a direct computation gives the first equality below.
\begin{align}\label{DeltaK}
\big(\Delta\mathbb Q_t^\beta\big)(x)
=\,&
-\frac{x\,\Delta Q_t^\beta(x)}
{\big(1+Q_t^\beta(x)\big)^2}
+
\frac{2x\big|\nabla Q_t^\beta(x)\big|^2}
{\big(1+Q_t^\beta(x)\big)^3}
-
\frac{2\nabla Q_t^\beta(x)}
{\big(1+Q_t^\beta(x)\big)^2}\nonumber\\
=\,&
-\frac{x\,\Delta Q_t^\beta(x)}
{\big(1+Q_t^\beta(x)\big)^2}
+
\frac{2x\big|\nabla Q_t^\beta(x)\big|^2}
{\big(1+Q_t^\beta(x)\big)^3}
+
\frac{2\big|\nabla Q_t^\beta(x)\big|}
{\big(1+Q_t^\beta(x)\big)^2}
\frac{x}{|x|} \nonumber\\
=\,&
-\frac{\Delta Q_t^\beta(x)}
{\big(1+Q_t^\beta(x)\big)^2}x
-
2\,\lambda_t^\beta(x)\,b_t^\beta(x).
\end{align}
The second equality follows from the fact that \(Q_t^\beta\) is radial and decreasing, so that \(\nabla Q_t^\beta(x) = -|\nabla Q_t^\beta(x)|\,\frac{x}{|x|}\), while the final equality holds by combining the last two terms as follows:
\begin{align}
2\,
\frac{
|x|\big|\nabla Q_t^\beta(x)\big|^2
+
\big(1+Q_t^\beta(x)\big)\big|\nabla Q_t^\beta(x)\big|
}
{\big(1+Q_t^\beta(x)\big)^3}
\frac{x}{|x|}
=\,&
2\,
\frac{
|x|\big|\nabla Q_t^\beta(x)\big|
+
1+Q_t^\beta(x)
}
{\big(1+Q_t^\beta(x)\big)^2}
\,\frac{\big|\nabla Q_t^\beta(x)\big|}{1+Q_t^\beta(x)} \frac{x}{|x|} \nonumber \\
=\,&-
2\,\lambda_t^\beta(x)\,b_t^\beta(x). \nonumber
\end{align}
where the last equality uses the definitions of \(b_t^\beta\) and
\(\lambda_t^\beta\) in~\eqref{DefDriftFun} and~\eqref{SigmaForm}, respectively.

Using the second-order chain rule identity~\eqref{Pip} with
\(\Psi=\overleftarrow{\mathbb Q}_t^\beta\), \(\Phi=\mathbb Q_t^\beta\), and
\(\nabla^\dagger \Phi=\nabla^\dagger\mathbb Q_t^\beta\), we obtain the first equality below.
\begin{align}
\big(\Hess\overleftarrow{\mathbb Q}_t^\beta\big)\big(\mathbb Q_t^\beta(x)\big)
=\,&
-
\Big(
\nabla^\dagger\mathbb Q_t^\beta(x)\otimes\big(\nabla^\dagger\mathbb Q_t^\beta(x)\big)^\dagger
\Big)^{-1}
\big(\Hess\mathbb Q_t^\beta\big)(x)
\Big(I_3\otimes\nabla^\dagger\mathbb Q_t^\beta(x)\Big)^{-1}
\nonumber\\
=\,&
-
\Big(
\big(\nabla^\dagger\mathbb Q_t^\beta(x)\big)^{-1}
\otimes
\big((\nabla^\dagger\mathbb Q_t^\beta(x))^\dagger\big)^{-1}
\Big)
\big(\Hess\mathbb Q_t^\beta\big)(x)
\Big(I_3\otimes\big(\nabla^\dagger\mathbb Q_t^\beta(x)\big)^{-1}\Big) \nonumber
\end{align}
The second equality uses \((B\otimes C)^{-1}=B^{-1}\otimes C^{-1}\) for invertible matrices \(B\) and \(C\). Using the above identity together with
\[
\big(I_3\otimes(\nabla^\dagger\mathbb Q_t^\beta(x))^\dagger\big)
\Big(
\big(\nabla^\dagger\mathbb Q_t^\beta(x)\big)^{-1}
\otimes
\big((\nabla^\dagger\mathbb Q_t^\beta(x))^\dagger\big)^{-1}
\Big)
=
\big(\nabla^\dagger\mathbb Q_t^\beta(x)\big)^{-1}\otimes I_3 ,
\]
and
\[
\Big(I_3\otimes\big(\nabla^\dagger\mathbb Q_t^\beta(x)\big)^{-1}\Big)
\big(I_3\otimes\nabla^\dagger\mathbb Q_t^\beta(x)\big)
=
\Big(I_3I_3\Big)
\otimes
\Big(
\big(\nabla^\dagger\mathbb Q_t^\beta(x)\big)^{-1}
\nabla^\dagger\mathbb Q_t^\beta(x)
\Big)
=
I_3\otimes I_3 ,
\]
we obtain the first equality below.
\begin{align}\label{I3}
\Tr_2
\Big[
\big(I_3\otimes(\nabla^\dagger\mathbb Q_t^\beta(x))^\dagger\big)
(\Hess\overleftarrow{\mathbb Q}_t^\beta)(\mathbb Q_t^\beta(x))
\big(I_3\otimes\nabla^\dagger\mathbb Q_t^\beta(x)\big)
\Big] 
& =
-\Tr_2
\Big[
\Big(
\big(\nabla^\dagger\mathbb Q_t^\beta(x)\big)^{-1}
\otimes I_3
\Big)
\big(\Hess\mathbb Q_t^\beta\big)(x)
\Big] \nonumber\\
& =
-\big(\nabla^\dagger\mathbb Q_t^\beta(x)\big)^{-1}
\Tr_2\big[(\Hess\mathbb Q_t^\beta)(x)\big] \nonumber\\
& =
-\big(\nabla^\dagger\mathbb Q_t^\beta(x)\big)^{-1}
\big(\Delta\mathbb Q_t^\beta\big)(x) \nonumber \\
&=\,
\big(\nabla^\dagger\mathbb Q_t^\beta(x)\big)^{-1}
\Bigg[
\frac{\Delta Q_t^\beta(x)}
{\big(1+Q_t^\beta(x)\big)^2}x
+
2\,\lambda_t^\beta(x)b_t^\beta(x)
\Bigg]
\nonumber\\
&=\,
\frac{1}{\lambda_t^\beta(x)}
\frac{\Delta Q_t^\beta(x)}
{\big(1+Q_t^\beta(x)\big)^2}x
+
2b_t^\beta(x)
\end{align}
The third equality uses that
\(\Tr_2[(\Hess\mathbb Q_t^\beta)(x)]=(\Delta\mathbb Q_t^\beta)(x)\), the fourth
equality uses~\eqref{DeltaK}, and the final equality uses that both \(x\) and
\(b_t^\beta(x)\) are radial vectors, and hence are eigenvectors of
\(\big(\nabla^\dagger\mathbb Q_t^\beta(x)\big)^{-1}\) with eigenvalue
\(\lambda_t^\beta(x)^{-1}\). Combining~\eqref{I3} and~\eqref{I1} and rearranging, we obtain~\eqref{PDEUp2}.
\end{proof}

\begin{proof}[Proof of Proposition~\ref{PropositionLocalizedSDEAwayOrigin}]
Recall that the function \(\overleftarrow{\mathbb Q}_t^\beta\) denotes the inverse of \(\mathbb Q_t^\beta\), so that $X_t=\overleftarrow{\mathbb Q}_{T-t}^\beta\big(Y_t^{T,\beta}\big)$. Since \(Y_t^{T,\beta}=\mathbb Q_{T-t}^\beta(X_t)\) and \(X_s\neq0\) for all \(s\in[a,b]\), the function $(t,y)\mapsto \overleftarrow{\mathbb Q}_{T-t}^\beta(y)$ is \(C^{1,2}\) along the path \(\{(t,Y_t^{T,\beta}):t\in[a,b]\}\) and It\^o's formula yields
\begin{align*}
dX_t
&=\,
d\overleftarrow{\mathbb Q}_{T-t}^\beta\big(Y_t^{T,\beta}\big) \\
&=\,
-\Big(\frac{\partial}{\partial t}\overleftarrow{\mathbb Q}_{T-t}^\beta\Big)
\big(Y_t^{T,\beta}\big)\,dt
+
\big(\nabla^\dagger\overleftarrow{\mathbb Q}_{T-t}^\beta\big)
\big(Y_t^{T,\beta}\big)\,dY_t^{T,\beta} \\
&\quad+
\frac12
\Tr_2
\Big[
\big(I_3\otimes(\nabla^\dagger\mathbb Q_{T-t}^{\beta}(X_t))^\dagger\big)
(\Hess\overleftarrow{\mathbb Q}_{T-t}^{\beta})(Y_t^{T,\beta})
\big(I_3\otimes\nabla^\dagger\mathbb Q_{T-t}^{\beta}(X_t)\big)
\Big]\,dt \\
&=\,
\big(\nabla^\dagger\overleftarrow{\mathbb Q}_{T-t}^{\beta}\big)
\big(Y_t^{T,\beta}\big)\,dY_t^{T,\beta}
+
b_{T-t}^{\beta}(X_t)\,dt ,
\end{align*}
where the last equality follows from~\eqref{PDEUp2}. Using \(Y_t^{T,\beta}=\mathbb Q_{T-t}^{\beta}(X_t)\) and Corollary~\ref{CorollaryMartI}, we obtain
\[
\big(\nabla^\dagger\overleftarrow{\mathbb Q}_{T-t}^{\beta}\big)
\big(Y_t^{T,\beta}\big)
\,dY_t^{T,\beta}
=
\big(\nabla^\dagger\overleftarrow{\mathbb Q}_{T-t}^{\beta}\big)
\big(\mathbb Q_{T-t}^{\beta}(X_t)\big)
\nabla^\dagger\mathbb Q_{T-t}^{\beta}(X_t)\,dW_t^{T,\beta}
=
dW_t^{T,\beta},
\]
where the second equality follows from~\eqref{PDEUp1}, and recall that \(W^{T,\beta}\) is the three-dimensional standard Brownian motion from Corollary~\ref{CorollaryMartI}. Substituting the above in the previous display yields the desired SDE~\eqref{SDEToSolve} on $[a,b]$.
\end{proof}

\section{Proof of the submartingale characterization}
\label{ProofSubMartCharacterization} 

The goal of this section is to prove Proposition~\ref{PropSubMart}. The following lemma, proved in Section~\ref{ProofLemmaKbounds1}, will be used to establish the square integrability of the continuous process \(\mathbf M^{T,\beta}\) defined in~\eqref{MartingaleFormula}.

\begin{lemma} \label{LemmaKbounds1}
Let \(T>0\), \(\beta>0\), and \(x\in\mathbb R^3\setminus\{0\}\). Then
\[
\int_0^T
\int_{\mathbb R^3}
\mathlarger p_{0,s}^{T,\beta}(x,y)
\,
\bigl|
 S_{T-s}^{\beta}(y)
b_{T-s}^{\beta}(y)
\bigr|^2
\,dy\,ds
<\infty .
\]
\end{lemma}

For \(\varepsilon>0\), let
\(\{\kappa_n^{\downarrow,\varepsilon}\}_{n\in\mathbb N}\) and
\(\{\kappa_n^{\uparrow,\varepsilon}\}_{n\in\mathbb N_0}\) be the sequences of
stopping times defined by \(\kappa_0^{\uparrow,\varepsilon}=0\) and, for
\(n\in\mathbb N\),
\begin{align}\label{VARRHOS}
\kappa_n^{\downarrow,\varepsilon}
:=\,&
\inf\big\{s\in(\kappa_{n-1}^{\uparrow,\varepsilon},T]:
X_s=0\big\}, \nonumber \\
\kappa_n^{\uparrow,\varepsilon}
:=\,&
\inf\big\{s\in(\kappa_n^{\downarrow,\varepsilon},T]:
|X_s|=\varepsilon\big\},
\end{align}
where we interpret \(\inf\emptyset=\infty\). 
Next, define the process
\(\chi^\varepsilon=\{\chi_s^\varepsilon\}_{s\in[0,T]}\) by
\[
\chi_s^\varepsilon
:=
\sum_{n=1}^{\infty}
\mathbf 1_{
[\kappa_{n-1}^{\uparrow,\varepsilon},
 \kappa_n^{\downarrow,\varepsilon})
}(s),
\qquad 0\le s\le T.
\]
Equivalently, \(\chi_s^\varepsilon=1\) whenever
\(s\in [\kappa_{n-1}^{\uparrow,\varepsilon}, \kappa_n^{\downarrow,\varepsilon})\) for some \(n\in\mathbb N\), and \(\chi_s^\varepsilon=0\) otherwise.

\begin{lemma}\label{LemmaMartingaleApproximationEpsilon}
Fix \(T,\beta>0\) and \(x\in\mathbb R^3\setminus\{0\}\). For \(\varepsilon>0\), set 
\begin{align}\label{MartingaleApproximationEpsilonFormula}
\mathbf M_t^{T,\beta,\varepsilon}
:=
 S_T^\beta(x)
-
\int_0^t
\chi_s^\varepsilon
\mathbf S_s^{T,\beta}
b_{T-s}^{\beta}(X_s)
\cdot dW_s^{T,\beta},
\qquad 0\le t\le T.
\end{align}
Then \(\mathbf M^{T,\beta,\varepsilon}\) is a square-integrable martingale and
\[
\mathbf E_x^{T,\beta}
\left[
\sup_{t\in[0,T]}
\left|
\mathbf M_t^{T,\beta,\varepsilon}
-
\mathbf M_t^{T,\beta}
\right|^2
\right]
\longrightarrow0,
\qquad
\varepsilon\downarrow0.
\]
\end{lemma}

\begin{proof}

Observe that the process \(\mathcal N_t^{T,\beta,\varepsilon}\), where $\mathcal N_t^{T,\beta,\varepsilon}
:=
\int_0^t
\chi_s^\varepsilon
\mathbf S_s^{T,\beta}
b_{T-s}^{\beta}(X_s)\cdot dW_s^{T,\beta}$ for $t\in[0,T]$, is a local martingale and, by the It\^o isometry,
\begin{align}\label{DipDap}
\mathbf E_x^{T,\beta}
\Big[
\big|\mathcal N_T^{T,\beta,\varepsilon}\big|^2
\Big]
=\,&
\mathbf E_x^{T,\beta}
\bigg[
\int_0^T
(\chi_s^\varepsilon)^2
\big(\mathbf S_s^{T,\beta}\big)^2
\big|b_{T-s}^{\beta}(X_s)\big|^2
\,ds
\bigg] \nonumber \\
\leq\,&
\mathbf E_x^{T,\beta}
\bigg[
\int_0^T
\big(\mathbf S_s^{T,\beta}\big)^2
\big|b_{T-s}^{\beta}(X_s)\big|^2
\,ds
\bigg]   \nonumber \\
=\,&
\int_0^T
\int_{\mathbb R^3}
\mathlarger p_{0,s}^{T,\beta}(x,y)
\big|
 S_{T-s}^{\beta}(y)
b_{T-s}^{\beta}(y)
\big|^2
\,dy\,ds
<\infty ,
\end{align}
where we have used that \(\chi_s^\varepsilon \le 1\), and the finiteness follows from Lemma~\ref{LemmaKbounds1}. Hence
\(\mathcal N^{T,\beta,\varepsilon}\) is a square-integrable martingale. Therefore, the process $\mathbf M^{T,\beta,\varepsilon} =  S_T^\beta(x) - \mathcal N^{T,\beta,\varepsilon}$ is also a square-integrable martingale. Taking \(\chi^\varepsilon\equiv1\) in the same argument shows that the process \(\mathbf M^{T,\beta}\), defined in~\eqref{MartingaleFormula}, is a square-integrable martingale. Thus, by Doob's inequality and the It\^o isometry,
\begin{align*}
\mathbf E_x^{T,\beta}
\bigg[
\sup_{t\in[0,T]}
\Big|
\mathbf M_t^{T,\beta,\varepsilon}
-
\mathbf M_t^{T,\beta}
\Big|^2
\bigg]
\leq\,&
4\,
\mathbf E_x^{T,\beta}
\bigg[
\Big|
\mathbf M_T^{T,\beta,\varepsilon}
-
\mathbf M_T^{T,\beta}
\Big|^2
\bigg]
\\
=\,&
4\,
\mathbf E_x^{T,\beta}
\Bigg[
\int_0^T
(1-\chi_s^\varepsilon)
\big(\mathbf S_s^{T,\beta}\big)^2
\big|b_{T-s}^{\beta}(X_s)\big|^2
\,ds
\Bigg] \nonumber \\
\le\,&
4
\int_0^T
\int_{\mathbb R^3}
\mathbf 1_{\{|y|\le\varepsilon\}}
\mathlarger p_{0,s}^{T,\beta}(x,y)
\big|
 S_{T-s}^{\beta}(y)
b_{T-s}^{\beta}(y)
\big|^2
\,dy\,ds,
\end{align*}
where in the equality we have also used that \(\chi_s^\varepsilon\in\{0,1\}\), the second inequality uses that \(1-\chi_s^\varepsilon \le \mathbf 1_{\{|X_s|\le\varepsilon\}}\). The above converges to \(0\) as \(\varepsilon\downarrow0\) by Lemma~\ref{LemmaKbounds1} and the dominated convergence theorem.
\end{proof}

\begin{remark}\label{RemarkMartingaleApproximationSum}
The process \(\mathbf M^{T,\beta,\varepsilon}\) given in~\eqref{MartingaleApproximationEpsilonFormula} admits the alternative representation
\begin{align}\label{MartingaleApproximationEpsilonFormula2}
\mathbf M_t^{T,\beta,\varepsilon}
=
 S_T^\beta(x)
+
\sum_{n=1}^{\infty}
\left(
\mathbf S_{t\wedge\kappa_n^{\downarrow,\varepsilon}}^{T,\beta}
-
\mathbf S_{t\wedge\kappa_{n-1}^{\uparrow,\varepsilon}}^{T,\beta}
\right),
\qquad 0\le t\le T.
\end{align}
Indeed, on each interval $[\kappa_{n-1}^{\uparrow,\varepsilon},
\kappa_n^{\downarrow,\varepsilon})$ the process \(X\) stays away from the origin. Hence, by
Proposition~\ref{PropositionLocalizedSDEAwayOrigin} and It\^o's formula,
\[
d\mathbf S_t^{T,\beta}
=
d S_{T-t}^{\beta}(X_t)
=
-\mathbf S_t^{T,\beta} b_{T-t}^{\beta}(X_t)
\cdot dW_t^{T,\beta}
\]
on this interval; see the application of It\^o's formula below~\eqref{ADEF}. Since \(\chi^\varepsilon\) is equal to \(1\) precisely on the intervals \([\kappa_{n-1}^{\uparrow,\varepsilon}, \kappa_n^{\downarrow,\varepsilon})\) and equal to \(0\) otherwise, we obtain
\begin{align*}
-
\int_0^t
\chi_s^\varepsilon
\mathbf S_s^{T,\beta}
b_{T-s}^{\beta}(X_s)
\cdot dW_s^{T,\beta}
=\,
\sum_{n=1}^{\infty}
\int_{0}^{t}
\mathbf 1_{
[\kappa_{n-1}^{\uparrow,\varepsilon},
 \kappa_n^{\downarrow,\varepsilon})
}(s)
\,d\mathbf S_s^{T,\beta}
=\,
\sum_{n=1}^{\infty}
\left(
\mathbf S_{t\wedge\kappa_n^{\downarrow,\varepsilon}}^{T,\beta}
-
\mathbf S_{t\wedge\kappa_{n-1}^{\uparrow,\varepsilon}}^{T,\beta}
\right).
\end{align*}
Substituting this identity into~\eqref{MartingaleApproximationEpsilonFormula}
gives~\eqref{MartingaleApproximationEpsilonFormula2}.
\end{remark}

\subsection{Proof of Proposition~\ref{PropSubMart}} \label{ProofPropSubMart}

\begin{proof}
Since \(0\leq  S_r^\beta(x)\leq1\) for all \(r\geq0\) and
\(x\in\mathbb R^3\), the process \(\mathbf S^{T,\beta}\) is bounded.
Moreover, since \(X\) is continuous and \((r,x)\mapsto S_r^\beta(x)\) extends continuously to \([0,\infty)\times\mathbb R^3\) by setting \( S_r^\beta(0):=0\), the process \(\mathbf S^{T,\beta}\) is continuous. To prove the submartingale property, let \(0\le s<t\le T\). By the Markov property of \(\mathbf P_x^{T,\beta}\), we have
\begin{align}\label{Check}
\mathbf E_x^{T,\beta}
\big[
\mathbf S_t^{T,\beta}
\,\big|\,
\mathcal F_s^{T,x}
\big]
&=
\mathbf E_x^{T,\beta}
\left[
 S_{T-t}^{\beta}(X_t)
\,\middle|\,
\mathcal F_s^{T,x}
\right] \nonumber\\
&=
\int_{\mathbb R^3}
\mathlarger p_{s,t}^{T,\beta}(X_s,y)
 S_{T-t}^{\beta}(y)
\,dy \nonumber\\
&=
 S_{T-s}^{\beta}(X_s)
\int_{\mathbb R^3}
P_{t-s}^{\beta}(X_s,y)
\,dy
\geq
 S_{T-s}^{\beta}(X_s)
=
\mathbf S_s^{T,\beta},
\end{align}
Here the third equality follows from definitions~\eqref{FirstTrans} and~\eqref{DefpFunction}, and the inequality uses that $\int_{\mathbb R^3} P_r^\beta(z,y)\,dy\geq1$ for all \(r>0\) and \(z\in\mathbb R^3\setminus\{0\}\). Since \(\mathbf S^{T,\beta}\) is a bounded continuous submartingale with respect to the augmented filtration \(\{\mathcal F_t^{T,x}\}_{t\in[0,T]}\), which satisfies the usual conditions, the Doob-Meyer decomposition theorem~\cite[Thm.~1.4.10]{Karatzas} applies. Consequently, there exist a continuous square-integrable martingale
\(\widetilde{\mathbf M}^{T,\beta}\) and a continuous predictable increasing process
\(\widetilde{\mathbf A}^{T,\beta}\), with
\(\widetilde{\mathbf A}_0^{T,\beta}=0\), such that
\[
\mathbf S_t^{T,\beta}
=
\widetilde{\mathbf M}_t^{T,\beta}
+
\widetilde{\mathbf A}_t^{T,\beta},
\qquad t\in[0,T].
\]
It remains to identify these two components explicitly. To this end, observe that taking \(\chi^\varepsilon\equiv1\) in the same argument in the proof of Lemma~\ref{LemmaMartingaleApproximationEpsilon}, shows that the process \(\mathbf M^{T,\beta}\), defined in~\eqref{MartingaleFormula}, is a square-integrable martingale.

Since $\mathbf S_t^{T,\beta}:= S_{T-t}^{\beta}(X_t)$ and
$\nabla S_t^\beta(y)=- S_t^\beta(y)b_t^\beta(y)$ for \(y\neq0\), it follows from~\eqref{MartingaleFormula} that the process $\mathbf A^{T,\beta}=\{\mathbf A_t^{T,\beta}\}_{t\in[0,T]}$, defined by $\mathbf A_t^{T,\beta}:=\mathbf S_t^{T,\beta}-\mathbf M_t^{T,\beta}$, can be expressed as
\begin{align}\label{ADEF}
\mathbf A_t^{T,\beta}
=\,&
 S_{T-t}^{\beta}(X_t)
-
 S_T^{\beta}(x)
-
\int_0^t
1_{\{X_s\neq0\}}
\big(\nabla S_{T-s}^{\beta}\big)(X_s)
\cdot dW_s^{T,\beta} \nonumber \\
=\,&
 S_{T-t}^{\beta}(X_t)
-
 S_T^{\beta}(x)
-
\int_0^t
\big(\nabla S_{T-s}^{\beta}\big)(X_s)
\cdot dW_s^{T,\beta},
\qquad
\mathbf P_x^{T,\beta}\text{-a.s.}
\end{align}
where the second equality holds by Remark~\ref{RemarkOriginNull}, after defining \(\nabla S_t^\beta(0):=0\). To prove that \(\mathbf A^{T,\beta}\) is constant on the excursions of \(X\) away from the origin, fix an interval \([a,b]\subset[0,T]\) such that
\(X_s\neq0\) for all \(s\in[a,b]\). 
Since \((t,y)\mapsto  S_{T-t}^{\beta}(y)\) is \(C^{1,2}\) away from
the origin, It\^o's formula gives, for \(t\in[a,b]\),
\begin{align*}
d S_{T-t}^{\beta}(X_t)
=\,&
-\Big(\frac{\partial}{\partial t}
 S_{T-t}^{\beta}\Big)(X_t)\,dt
+
\big(\nabla S_{T-t}^{\beta}\big)(X_t)\cdot dX_t
+
\frac12
\big(\Delta S_{T-t}^{\beta}\big)(X_t)\,dt  \\
=\,&
\Bigg[
-\Big(\frac{\partial}{\partial t}
 S_{T-t}^{\beta}\Big)(X_t)
+
\big(\nabla S_{T-t}^{\beta}\big)(X_t)\cdot
b_{T-t}^{\beta}(X_t)
+
\frac12
\big(\Delta S_{T-t}^{\beta}\big)(X_t)
\Bigg]dt  \nonumber \\
+ \,&
\big(\nabla S_{T-t}^{\beta}\big)(X_t)\cdot dW_t^{T,\beta} \nonumber \\
=\,&
\big(\nabla S_{T-t}^{\beta}\big)(X_t)\cdot dW_t^{T,\beta},
\end{align*}
where second equality uses that the process \(X\) satisfies the SDE~\eqref{SDEToSolve} on $[a,b]$ by Proposition~\ref{PropositionLocalizedSDEAwayOrigin} and the final equality uses that the term in square brackets is zero by~\eqref{PartialForP}. Therefore, using the identity above and writing~\eqref{ADEF} in differential form, we obtain
\[
d\mathbf A_t^{T,\beta}
=
d S_{T-t}^{\beta}(X_t)
-
\big(\nabla S_{T-t}^{\beta}\big)(X_t)
\cdot dW_t^{T,\beta}
=
0,
\qquad t\in[a,b].
\]
Thus, \(\mathbf A^{T,\beta}\) is constant on \([a,b]\).

It remains to prove that \(\mathbf A^{T,\beta}\) is increasing. For
\(\varepsilon>0\), let
\(\{\kappa_n^{\downarrow,\varepsilon}\}_{n\in\mathbb N}\) and
\(\{\kappa_n^{\uparrow,\varepsilon}\}_{n\in\mathbb N_0}\) be the stopping
times defined in~\eqref{VARRHOS}. Define the process
\(\mathbf A^{T,\beta,\varepsilon}
=
\{\mathbf A_t^{T,\beta,\varepsilon}\}_{t\in[0,T]}\)
by
\[
\mathbf A_t^{T,\beta,\varepsilon}
:=
\sum_{n=1}^{\infty}
\left(
\mathbf S_{t\wedge\kappa_n^{\uparrow,\varepsilon}}^{T,\beta}
-
\mathbf S_{t\wedge\kappa_n^{\downarrow,\varepsilon}}^{T,\beta}
\right)
=
\sum_{n=1}^{\infty}
\mathbf S_{t\wedge\kappa_n^{\uparrow,\varepsilon}}^{T,\beta}
\mathbf 1_{\{\kappa_n^{\downarrow,\varepsilon}\le t\}},
\]
where the second equality uses that $\mathbf S_{\kappa_n^{\downarrow,\varepsilon}}^{T,\beta} =  S_{T-\kappa_n^{\downarrow,\varepsilon}}^\beta(0) = 0$ on \(\{\kappa_n^{\downarrow,\varepsilon}\le T\}\), consequently, $\mathbf S_{t\wedge\kappa_n^{\downarrow,\varepsilon}}^{T,\beta} = 0$ whenever \(\kappa_n^{\downarrow,\varepsilon}\le t\), while $\mathbf 1_{\{\kappa_n^{\downarrow,\varepsilon}\le t\}}=0$ otherwise. Let \(\mathbf M^{T,\beta,\varepsilon}\) be the process  given in~\eqref{MartingaleApproximationEpsilonFormula2}, then we have
\[
\mathbf M_t^{T,\beta}
+
\mathbf A_t^{T,\beta}
=
\mathbf S_t^{T,\beta}
=
\mathbf M_t^{T,\beta,\varepsilon}
+
\mathbf A_t^{T,\beta,\varepsilon},
\qquad 0\le t\le T.
\]
Consequently, $\mathbf A_t^{T,\beta,\varepsilon} - \mathbf A_t^{T,\beta} = \mathbf M_t^{T,\beta} - \mathbf M_t^{T,\beta,\varepsilon}$ for $0\le t\le T$. Therefore, by Lemma~\ref{LemmaMartingaleApproximationEpsilon}, we have
\[
\mathbf E_x^{T,\beta}
\left[
\sup_{t\in[0,T]}
\left|
\mathbf A_t^{T,\beta,\varepsilon}
-
\mathbf A_t^{T,\beta}
\right|^2
\right]
\longrightarrow0,
\qquad
\varepsilon\downarrow0.
\]
In particular, along a sequence \(\varepsilon_k\downarrow0\), we may assume that
\[
\sup_{t\in[0,T]}
\left|
\mathbf A_t^{T,\beta,\varepsilon_k}
-
\mathbf A_t^{T,\beta}
\right|
\longrightarrow0,
\qquad
\mathbf P_x^{T,\beta}\text{-a.s.}
\]

Fix \(\delta\in(0,T)\). For \(u\in[0,T-\delta]\), whenever
\(\mathbf A_u^{T,\beta,\varepsilon}\) decreases from its previous running
maximum, the decrease can only occur during one of the intervals
\(
[\kappa_n^{\downarrow,\varepsilon},
 \kappa_n^{\uparrow,\varepsilon})
\),
on which \(|X_u|<\varepsilon\). On these intervals, we have
\[
\mathbf S_u^{T,\beta}
=
 S_{T-u}^{\beta}(X_u)
=
\frac{1}{1+Q_{T-u}^\beta(X_u)}
=
\frac{1}{1+\bar Q_{T-u}^\beta(|X_u|)}
\le
\frac{1}{1+\bar Q_\delta^\beta(\varepsilon)},
\]
where recall that \(\bar Q_t^\beta\) denotes the radial profile of \(Q_t^\beta\), and the inequality follows from \(|X_u|<\varepsilon\), \(T-u\ge\delta\), together with the facts that \(r\mapsto \bar Q_t^\beta(r)\) is decreasing, and \(t\mapsto \bar Q_t^\beta(r)\) is increasing. Hence, for
\(0\le s\le t\le T-\delta\),
\[
\mathbf A_t^{T,\beta,\varepsilon}
\ge
\mathbf A_s^{T,\beta,\varepsilon}
-
\frac{1}{1+\bar Q_\delta^\beta(\varepsilon)}.
\]
Also, by~\eqref{OnePlusQBetaBoundsOld}, we have
\[
0
\le
\frac{1}{1+\bar Q_\delta^\beta(\varepsilon)}
\le
\frac{\varepsilon\sqrt{\pi}}{\sqrt{\delta}}
e^{\frac{\varepsilon^2}{4\delta}}
\longrightarrow 0,
\qquad
\varepsilon\downarrow0.
\]
Therefore, passing to the sequence \(\varepsilon_k\downarrow0\) in the preceding inequality
and using the uniform convergence
\(\mathbf A^{T,\beta,\varepsilon_k}\to\mathbf A^{T,\beta}\), we obtain
\[
\mathbf A_t^{T,\beta}
\ge
\mathbf A_s^{T,\beta},
\qquad
0\le s\le t\le T-\delta,
\qquad
\mathbf P_x^{T,\beta}\text{-a.s.}
\]
Thus \(\mathbf A^{T,\beta}\) is increasing on \([0,T-\delta]\). Since
\(\delta\in(0,T)\) was arbitrary and \(\mathbf A^{T,\beta}\) is continuous, it
follows that \(\mathbf A^{T,\beta}\) is increasing on \([0,T]\).
\end{proof}

\section{The probability of visiting the  origin}

Let \(\mathcal O_T\) denote the event consisting of all outcomes \(\omega\in\Omega\) such that the corresponding path visits the origin. Then \(\mathcal O_T=\{\tau\le T\}\).

\subsection{Proof of Theorem~\ref{ThmSubMART}}\label{SubsectionThmSubMART}

\begin{proof}

\noindent Part (i). Recall from Proposition~\ref{PropSubMart} that
\(\mathbf S^{T,\beta} = \mathbf M^{T,\beta} + \mathbf A^{T,\beta}\), where \(\mathbf M^{T,\beta}\) is a continuous martingale and the continuous increasing component \(\mathbf A^{T,\beta}\) remains constant on every interval on which \(X\) does not hit the origin. Since \(x\neq0\), the process \(X\) does not hit the origin on \([0,\tau)\), and
$\mathbf A_0^{T,\beta} = \mathbf S_0^{T,\beta} - \mathbf M_0^{T,\beta} =0$, we have
$\mathbf A_t^{T,\beta}=0$ for $0\le t<\tau$. Consequently,
$\mathbf S_t^{T,\beta} = \mathbf M_t^{T,\beta}$ for $0\le t<\tau,$
and hence the stopped process $\big\{
\mathbf S_{t\wedge\tau}^{T,\beta}
\big\}_{t\in[0,T]}$ is a bounded martingale. Therefore, by the optional stopping theorem,
\[
 S_T^\beta(x)
=
\mathbf S_0^{T,\beta}
=
\mathbf E_x^{T,\beta}
\big[
\mathbf S_{\tau\wedge T}^{T,\beta}
\big].
\]
Next, on the event \(\{\tau\le T\}\), we have \(X_\tau=0\) and so
\(\mathbf S_{\tau\wedge T}^{T,\beta} =  S_{T-\tau}^{\beta}(0) = 0\).
On the other hand, on
\(\{\tau>T\}\), we have
\(\mathbf S_{\tau\wedge T}^{T,\beta} =  S_0^\beta(X_T) = 1\).
Consequently,
\begin{align}\label{1OComp}
\mathbf S_{\tau\wedge T}^{T,\beta}
=
\mathbf 1_{\{\tau>T\}}
=
\mathbf 1_{\mathcal O_T^c},
\qquad
\mathbf P_x^{T,\beta}\text{-a.s.}
\end{align}
Substituting~\eqref{1OComp} into the previous display and recalling that $ S_T^\beta(x) =\big(1+Q_T^\beta(x)\big)^{-1}$ yields the desired result. \vspace{.2cm}

\noindent Part (ii). Define $\widetilde{\mathbf P}_x^{T,\beta} := \mathbf P_x^{T,\beta} \big[\,\cdot\,\big|\,\mathcal O_T^c\big]$. Recall from the proof of Part~(i) that \(\{\mathbf S_{t\wedge\tau}^{T,\beta}\}_{t\in[0,T]}\) is a bounded
martingale. Hence, for every \(t\in[0,T]\), we have
\[
\mathbf E_x^{T,\beta}
\big[
\mathbf 1_{\mathcal O_T^c}
\,\big|\,
\mathcal F_t^{T,x}
\big]
\overset{\eqref{1OComp}}{=}
\mathbf E_x^{T,\beta}
\big[
\mathbf S_{\tau\wedge T}^{T,\beta}
\,\big|\,
\mathcal F_t^{T,x}
\big]
=
\mathbf S_{t\wedge\tau}^{T,\beta}.
\]
Thus the density process of \(\widetilde{\mathbf P}_x^{T,\beta}\) with respect to
\(\mathbf P_x^{T,\beta}\) is
\[
Z_t^{T,\beta}
:=
\frac{
d\widetilde{\mathbf P}_x^{T,\beta}
}{
d\mathbf P_x^{T,\beta}
}
\bigg|_{\mathcal F_t^{T,x}}
=
\frac{
\mathbf E_x^{T,\beta}
\big[
\mathbf 1_{\mathcal O_T^c}
\,\big|\,
\mathcal F_t^{T,x}
\big]
}{
\mathbf P_x^{T,\beta}[\mathcal O_T^c]
}
=
\big(1+Q_T^\beta(x)\big)\,\mathbf S_{t\wedge\tau}^{T,\beta},
\qquad 0\le t\le T,
\]
where the last equality also uses Part~(i). The process
\(\{Z_t^{T,\beta}\}_{t\in[0,T]}\) is a bounded continuous martingale, with \(Z_0^{T,\beta}=1\).
Hence, for \(0\le t\le T\), we have
\begin{align}\label{Emmy}
dZ_t^{T,\beta}
=\,&
\big(1+Q_T^\beta(x)\big)
d\mathbf S_{\tau\wedge t}^{T,\beta}
\nonumber\\
=\,&
-
\mathbf 1_{\{t<\tau\}}
\big(1+Q_T^\beta(x)\big)
\mathbf S_t^{T,\beta}
b_{T-t}^\beta(X_t)\cdot dW_t^{T,\beta}
=\,
-
Z_t^{T,\beta}
b_{T-t}^\beta(X_t)\cdot dW_t^{T,\beta},
\end{align}
where the second equality follows from~\eqref{SDEforSuptoTau}, and the last equality uses \(Z_t^{T,\beta} = \mathbf 1_{\{t<\tau\}}
\big(1+Q_T^\beta(x)\big)\mathbf S_t^{T,\beta}\). It remains to verify L\'evy's characterization under
\(\widetilde{\mathbf P}_x^{T,\beta}\).

\noindent\textit{The process \(\{X_t\}_{t\in[0,T]}\) is a
\(\widetilde{\mathbf P}_x^{T,\beta}\)-martingale.}
For \(0\le t\le T\), applying It\^o's product formula to the
\(\mathbb R^3\)-valued process
\(\{Z_t^{T,\beta}X_t\}_{t\in[0,T]}\), we obtain
\begin{align}\label{ZX}
d\big(Z_t^{T,\beta}X_t\big)
=\,&
X_t\,dZ_t^{T,\beta}
+
Z_t^{T,\beta}\,dX_t
+
dX_t\,dZ_t^{T,\beta}.
\end{align}
By definition, \(Z_t^{T,\beta}=0\) for \(t\in[\tau,T]\), and hence
\(dZ_t^{T,\beta}=0\) on \([\tau,T]\). Therefore it suffices to analyze
\eqref{ZX} on the stochastic interval \([0,\tau)\). On this interval,
the process \(X\) has not yet reached the origin, and consequently,
using~\eqref{SDEToSolve} and~\eqref{Emmy}, we obtain
\begin{align*}
dX_t\,dZ_t^{T,\beta}
&=
\Big(
dW_t^{T,\beta}
+
b_{T-t}^{\beta}(X_t)\,dt
\Big)
\Big(
-
Z_t^{T,\beta}
b_{T-t}^{\beta}(X_t)\cdot dW_t^{T,\beta}
\Big)
\\
&=
-
Z_t^{T,\beta}
b_{T-t}^{\beta}(X_t)\,dt
=
Z_t^{T,\beta}dW_t^{T,\beta}
-
Z_t^{T,\beta}dX_t,
\end{align*}
where the last equality again follows from~\eqref{SDEToSolve}.
Substituting this identity into~\eqref{ZX}, we obtain, on
\([0,\tau)\),
\begin{align}\label{ZX2}
d\big(Z_t^{T,\beta}X_t\big)
=
X_t\,dZ_t^{T,\beta}
+
Z_t^{T,\beta}\,dW_t^{T,\beta}.
\end{align}
Since \(Z_t^{T,\beta}X_t=0\) for \(t\in[\tau,T]\), the process
\(\{Z_t^{T,\beta}X_t\}_{t\in[0,T]}\) is constant on \([\tau,T]\).
It follows from~\eqref{ZX2} that the process
\(\{Z_t^{T,\beta}X_t\}_{t\in[0,T]}\) is a local
\(\mathbf P_x^{T,\beta}\)-martingale. Since \(Z^{T,\beta}\) is bounded
and the second moments of \(X_t\) were shown to be finite in the proof
of Lemma~\ref{LemmaL2IncrementY}, it follows that
\(\{Z_t^{T,\beta}X_t\}_{t\in[0,T]}\) is square-integrable, and hence a
\(\mathbf P_x^{T,\beta}\)-martingale. Since \(Z^{T,\beta}\) is the
density process of \(\widetilde{\mathbf P}_x^{T,\beta}\) with respect to
\(\mathbf P_x^{T,\beta}\), Bayes' rule then yields that
\(\{X_t\}_{t\in[0,T]}\) is a
\(\widetilde{\mathbf P}_x^{T,\beta}\)-martingale. \vspace{.2cm}

\noindent \textit{The process \(\{(v\cdot X_t)^2-t\}_{t \in [0, T]}\) is a \(\widetilde{\mathbf P}_x^{T,\beta}\)-martingale for any unit vector \(v\in\mathbb R^3\).} Let \(v\in\mathbb R^3\) be a unit vector and  define the processes
\[
X_t^v:=v\cdot X_t,
\qquad
W_t^{T,\beta,v}:=v\cdot W_t^{T,\beta}.
\]
Recall that \(Z_t^{T,\beta}=0\) and \(dZ_t^{T,\beta}=0\) on \([\tau,T]\); hence all \(Z^{T,\beta}\)-weighted differentials below vanish on \([\tau,T]\).
Therefore it suffices to carry out the computation on the stochastic
interval \([0,\tau)\). On this interval, Proposition~\ref{PropositionLocalizedSDEAwayOrigin} yields
\[
dX_t^v
=
dW_t^{T,\beta,v}
+
v\cdot b_{T-t}^{\beta}(X_t)\,dt,
\qquad
dW_t^{T,\beta,v}\,dW_t^{T,\beta}
=
v\,dt .
\]
Applying It\^o's product formula to
\(\{(X_t^v)^2Z_t^{T,\beta}\}_{t\in[0,T]}\) and using the above differential identities together with~\eqref{Emmy}, we obtain, on \([0,\tau)\),
\begin{align*}
d\Big((X_t^v)^2Z_t^{T,\beta}\Big)
=\,&
Z_t^{T,\beta}\,d\big((X_t^v)^2\big)
+
(X_t^v)^2\,dZ_t^{T,\beta}
+
d\big((X_t^v)^2\big)dZ_t^{T,\beta}
\\
=\,&
Z_t^{T,\beta}
\Big(
2X_t^v\,dW_t^{T,\beta,v}
+
2X_t^v\,v\cdot b_{T-t}^{\beta}(X_t)\,dt
+
dt
\Big)
\\
+\,&
(X_t^v)^2\,dZ_t^{T,\beta}
-
2Z_t^{T,\beta}X_t^v\,v\cdot b_{T-t}^{\beta}(X_t)\,dt
\\
=\,&
2Z_t^{T,\beta}X_t^v\,dW_t^{T,\beta,v}
+
(X_t^v)^2\,dZ_t^{T,\beta}
+
Z_t^{T,\beta}\,dt .
\end{align*}
Here the last equality follows by cancellation of the drift terms. Since all
terms vanish on \([\tau,T]\), the above identity determines the
differential on the whole interval \([0,T]\), and consequently the process
\(\big\{\big((X_t^v)^2-t\big)Z_t^{T,\beta}\big\}_{t\in[0,T]}\) is a local
\(\mathbf P_x^{T,\beta}\)-martingale. As before, the boundedness of
\(Z^{T,\beta}\) and the finite second moments of \(X_t\) imply that this
local martingale is a \(\mathbf P_x^{T,\beta}\)-martingale, and consequently
Bayes' rule yields that
\(\big\{(X_t^v)^2-t\big\}_{t\in[0,T]}\) is a
\(\widetilde{\mathbf P}_x^{T,\beta}\)-martingale. Since we have already shown
that \(\{X_t\}_{t\in[0,T]}\) is a continuous
\(\widetilde{\mathbf P}_x^{T,\beta}\)-martingale, L\'evy's characterization
theorem implies that \(\{X_t\}_{t\in[0,T]}\) is a three-dimensional Brownian
motion started from \(x\) under \(\widetilde{\mathbf P}_x^{T,\beta}\). \vspace{.2cm}

\noindent Part (iii). Let \(A\in\mathcal F_{\sigma}^{T,x}\). Since the conditional law $\widetilde{\mathbf P}_x^{T,\beta} =
\mathbf P_x^{T,\beta} \big[\,\cdot\,\big|\,\mathcal O_T^c\big]$ coincides with the three-dimensional Wiener measure \(\mathbf P_x\) by Part~(ii), we obtain the second equality below
\begin{align} \label{EqRightMeasure}
\mathbf P_x^{T,\beta}
\big[
A\cap\mathcal O_T^c
\big]
&=
\mathbf P_x^{T,\beta}
\big[
A\,\big|\,\mathcal O_T^c
\big]
\,
\mathbf P_x^{T,\beta}
\big[
\mathcal O_T^c
\big]
=
\mathbf P_x[A]\,\mathbf P_x^{T,\beta}[\mathcal O_T^c]
=
\frac{1}{1+Q_T^\beta(x)}
\mathbf P_x[A]
\end{align}
Here the last equality uses Part~(i). On the other hand, by the Markov property, we have
\begin{align}
\mathbf 1_{\{\sigma<\tau\}}
\mathbf P_x^{T,\beta}
\big[
\mathcal O_T^c
\,\big|\,
\mathcal F_{\sigma}^{T,x}
\big]
&=
\mathbf 1_{\{\sigma<\tau\}}
\mathbf P_{X_{\sigma}}^{T-\sigma,\beta}
\big[
\mathcal O_{T-\sigma}^c
\big]
=
\mathbf 1_{\{\sigma<\tau\}}
\frac{1}{1+Q_{T-\sigma}^{\beta}(X_{\sigma})}. \nonumber
\end{align}
Using the inclusion \(\mathcal O_T^c\subseteq\{\sigma<\tau\}\) and the preceding identity, we obtain the final equality below.
\begin{align}
\frac{1}{1+Q_T^\beta(x)}
\mathbf P_x[A]
&\overset{\eqref{EqRightMeasure}}{=} 
\mathbf P_x^{T,\beta}
\big[
A\cap\mathcal O_T^c
\big] \nonumber \\
&=
\mathbf E_x^{T,\beta}
\left[
\mathbf 1_A
\mathbf P_x^{T,\beta}
\big[
\mathcal O_T^c
\,\big|\,
\mathcal F_{\sigma}^{T,x}
\big]
\right] 
=
\mathbf E_x^{T,\beta}
\left[
\mathbf 1_{A\cap\{\sigma<\tau\}}
\frac{1}{1+Q_{T-\sigma}^{\beta}(X_{\sigma})}
\right] \nonumber
\end{align}
Since \(1+Q_{T-\sigma}^{\beta}(X_{\sigma})\) is strictly positive and \(\mathcal F_{\sigma}^{T,x}\)-measurable, the preceding equality identifies \(\big(1+Q_{T-\sigma}^{\beta}(X_{\sigma})\big)^{-1}\) as the Radon--Nikodym derivative of the measure \(B\mapsto \mathbf P_x[B]/(1+Q_T^\beta(x))\) with respect to the measure \(B\mapsto \mathbf P_x^{T,\beta}[B\cap\{\sigma<\tau\}]\). Hence,
\[
\mathbf P_x^{T,\beta}
\big[
A\cap\{\sigma<\tau\}
\big]
=
\frac{1}{1+Q_T^\beta(x)}
\mathbf E_x
\left[
\mathbf 1_A
\big(
1+Q_{T-\sigma}^{\beta}(X_{\sigma})
\big)
\right].
\]
In particular, taking \(A=\Omega\) in the above equality gives
\(\mathbf P_x^{T,\beta}[\sigma<\tau] = \big(1+Q_T^\beta(x)\big)^{-1} \mathbf E_x \big[ 1+Q_{T-\sigma}^{\beta}(X_{\sigma}) \big]\). Using this and the preceding identity in the definition of the conditional probability \(\mathbf P_x^{T,\beta}\big[\,\cdot\,\big|\,\sigma<\tau\big]\) yields the desired identity. \vspace{.2cm}

\noindent Part (iv). Let \(s\in[0,T]\) be fixed. Since
\(\mathcal O_T^c\subseteq\{\tau>s\}\), using~\eqref{1OComp} we obtain the first equality below.
\begin{align}
\mathbf E_x^{T,\beta}
\Big[
\mathbf S_{\tau\wedge T}^{T,\beta}
\Big]
&=
\mathbf E_x^{T,\beta}
\Big[
\mathbf S_{\tau\wedge T}^{T,\beta}
1_{\{\tau>s\}}
\Big] \nonumber \\
&=\mathbf E_x^{T,\beta}
\Big[
\mathbf E_x^{T,\beta}
\Big[
\mathbf S_{\tau\wedge T}^{T,\beta}
\,\Big|\,
\mathcal F_s^{T,x}
\Big]
1_{\{\tau>s\}}
\Big]
\nonumber\\
&=
\mathbf E_x^{T,\beta}
\Big[
\mathbf E_{X_s}^{T-s,\beta}
\Big[
\mathbf S_{\tau\wedge(T-s)}^{T-s,\beta}
\Big]
1_{\{\tau>s\}}
\Big]
\nonumber\\
&=
\mathbf E_x^{T,\beta}
\Big[
S_{T-s}^\beta(X_s)
1_{\{\tau>s\}}
\Big] \nonumber \\
&=
\mathbf P_x^{T,\beta}[\tau>s]\,
\mathbf E_x^{T,\beta}
\Big[
S_{T-s}^\beta(X_s)
\,\Big|\,
\tau>s
\Big]
\nonumber \\
&=
\mathbf P_x^{T,\beta}[\tau>s]\,
\frac{
1
}{
\int_{\mathbb R^3}
P_s(x,y)
\big(1+Q_{T-s}^\beta(y)\big)\,dy
} \nonumber
\end{align}
Here the third equality follows from the Markov property, the fourth equality holds by the same reasoning as the first, the fifth equality follows from the definition of conditional expectation, and the final equality applies Part~(iii) with the nonrandom stopping time \(\sigma=s\). Using the inclusion \(\{\tau>T\}\subseteq\{\tau>s\}\) and the expression for \(\mathbf P_x^{T,\beta}[\tau>s]\) obtained from above, we get
\begin{align}
\mathbf P_x^{T,\beta}
\big[
\tau>s
\,\big|\,
\tau\le T
\big]
&=
\frac{
\mathbf P_x^{T,\beta}[\tau>s]
-
\mathbf P_x^{T,\beta}[\tau>T]
}{
\mathbf P_x^{T,\beta}[\tau\le T]
}
\nonumber\\
&=
\frac{
\mathbf E_x^{T,\beta}
\big[
\mathbf S_{\tau\wedge T}^{T,\beta}
\big]\int_{\mathbb R^3} P_s(x,y)
\bigl(1+Q_{T-s}^{\beta}(y)\bigr)\,dy
-
\mathbf P_x^{T,\beta}[\tau > T]
}{
\mathbf P_x^{T,\beta}[\tau\le T]
} . \nonumber
\end{align}
Finally, using~\eqref{1OComp} together with Part~(i), we obtain the desired result.
\end{proof}

\section{Proof of auxiliary lemmas} \label{ProofAuxLemmas}

For each \(\beta>0\) and each \(T>0\), by~\cite[Lem.~8]{Fleischmann} there exists a constant \(C=C(\beta,T)>0\) such that for all \(0<t\le T\) and all \(x,y\in\R^3\setminus\{0\}\), 
\begin{align}\label{PtBetaBounds}
P_t(x,y)
\;\le\;
P_t^\beta(x,y)
\;\le\;
P_t(x,y)
+
C\,t^{-1/2}\,
\frac{e^{-\frac{|x|^2}{4t}}}{|x|}
\frac{e^{-\frac{|y|^2}{4t}}}{|y|}\,,
\end{align}
where \(P_t^\beta(x,y)\) is defined in~\eqref{DefPointKer3dBeta} and \(P_t(x,y)\) denotes the three-dimensional free heat kernel. For every \(L>0\), there exists a constant \(C_L>0\) such that, for all \(T,\beta>0\) with \(\beta\sqrt T\le L\) and all
\(x\in\R^3\setminus\{0\}\), the function $Q_T^\beta(x)$ defined in~\eqref{DefPointKer3dBeta} has the following bounds from~\cite[Lem.~5.5]{Mian2}.
\begin{align}
\frac{\sqrt{T}}{|x|\sqrt{\pi}}
e^{-\frac{|x|^2}{4T}}
\,\le\, &
1+Q_T^\beta(x)
\le
C_L\Bigl(1+\frac{\sqrt T}{|x|}\Bigr)e^{-4\pi\beta|x|}
\,, \label{OnePlusQBetaBoundsOld}
\\
\big|\nabla Q_T^\beta(x)\big|
\,\le\, &
C_L
\left[
\frac{1}{|x|}
+\frac{1+\sqrt{T}}{|x|^2}
+\frac{1}{|x|\sqrt{T}}
\right]
e^{-4\pi\beta|x|}
\,. \label{GradQBetaBoundOld}
\end{align}

\begin{lemma}[Small-time Gaussian estimates for \(Q_t^\beta\)]
\label{LemSmallTimeGaussianEstimate}
Fix \(T,\beta>0\). There exists \(C=C(T,\beta)>0\) such that, for all
\(0<t\le T/2\) and all \(x\in\mathbb R^3\setminus\{0\}\)
\begin{enumerate}[(i)]
    \item $Q_t^\beta(x) \le    C\frac{\sqrt t}{|x|}    e^{-\frac{|x|^2}{16t}}$,
    \item $    |\nabla_x Q_t^\beta(x)| \le C     \left(    \frac1{|x|}    +    \frac{\sqrt t}{|x|^2}  \right)     e^{-\frac{|x|^2}{16t}}$.
\end{enumerate}
\end{lemma}

\begin{proof}
Recalling that $P_t(r) = (4\pi t)^{-3/2}
e^{-r^2/(4t)}$ denotes the radial form of the three-dimensional free heat kernel, it follows from~\eqref{DefH} that
\begin{align}\label{DefH2nd}
Q_t^\beta(x)
=
\frac{1}{4\pi\beta |x|\sqrt{\pi t}}
\int_0^\infty
\bigl(e^{4\pi\beta w}-1\bigr)
e^{-\frac{(|x|+w)^2}{4t}}
\,dw.
\end{align}

\noindent Part (i): Set \(r=|x|\).
Observe that for \(w\ge0\) and \(0<t\le T/2\), we have
\begin{align}
4\pi\beta w-\frac{w^2}{8t}
=
-\frac{\bigl(w-16\pi\beta t\bigr)^2}{8t}
+
32\pi^2\beta^2t
\le
32\pi^2\beta^2t
\le
16\pi^2\beta^2T ,
\nonumber
\end{align}
where the final inequality uses \(0<t\le T/2\). Therefore,
\begin{align}
e^{4\pi\beta w}
\le
e^{16\pi^2\beta^2T}
e^{\frac{w^2}{8t}}
\le
e^{16\pi^2\beta^2T}
e^{\frac{(r+w)^2}{8t}},
\nonumber
\end{align}
where the second inequality uses \(w^2\le (r+w)^2\). Using the elementary inequality
\(e^a-1\le ae^a\) for \(a\ge0\) with \(a=4\pi\beta w\), together with the above estimate, we obtain
\begin{align} \label{AnInequinProof2}
e^{4\pi\beta w}-1
\le
4\pi\beta w\,e^{4\pi\beta w}
\le
4\pi\beta w\,e^{16\pi^2\beta^2T}
e^{\frac{(r+w)^2}{8t}}
=
Cw\,e^{\frac{(r+w)^2}{8t}},
\end{align}
where \(C=4\pi\beta e^{16\pi^2\beta^2T}\). Substituting this estimate into~\eqref{DefH2nd} and absorbing the prefactor into the constant \(C\), we obtain the inequality below,
\[
Q_t^\beta(x)
\le
\frac{C}{r\sqrt t}
\int_0^\infty
w
e^{\frac{(r+w)^2}{8t}}
e^{-\frac{(r+w)^2}{4t}}
dw
=
\frac{C}{r\sqrt t}
\int_0^\infty
w
e^{-\frac{(r+w)^2}{8t}}
dw
=
\frac{C}{r\sqrt t}
\int_r^\infty
(u-r)e^{-\frac{u^2}{8t}}du,
\]
where the second equality uses the change of variables \(u=r+w\). Since the last integral is bounded by $\int_r^\infty u e^{-\frac{u^2}{8t}} du = 4t e^{-\frac{r^2}{8t}}$, combining the constants and using $e^{-\frac{r^2}{8t}} \le e^{-\frac{r^2}{16t}}$ yields the desired bound. \vspace{.2cm}

\noindent Part (ii): Let \(\bar Q_t^\beta:[0,\infty)\to\mathbb R\) denote the radial profile of \(Q_t^\beta\), defined by $\bar Q_t^\beta(r):=Q_t^\beta(x)$ for $|x|=r$ and all $r\in[0,\infty)$. Then, using that $\big| \nabla_x Q_t^\beta(x)\big| = -\frac{\partial}{\partial r}\bar{Q}_t^\beta(r)$ with $r=|x|$, and recalling~\eqref{DefH2nd}, we obtain
\begin{align}
\big| \nabla_x Q_t^\beta(x)\big|
\,=\,&
-\frac{\partial}{\partial r}\,
\frac{1}{4\pi\beta \, r\, \sqrt{\pi t}}
\int_{0}^{\infty}\big(e^{4\pi\beta w}\,-\, 1\big)\,
e^{-\frac{(r+w)^2}{4t}}\,dw
\nonumber\\
\,=\,&
\frac{1}{r}\, Q_t^{\beta}(x)
\,+\,
\frac{C}{ r \sqrt{t}}
\int_{0}^{\infty}\big(e^{4\pi\beta w}\, -\, 1\big)\,
\frac{r+w}{t}\,
e^{-\frac{(r+w)^2}{4t}}\,dw \nonumber \\
\,\le \, &
C\frac{\sqrt t}{r^2} e^{-\frac{r^2}{16t}}
\,+\,
\frac{C}{r t^{3/2}}
\int_0^\infty
w
e^{\frac{(r+w)^2}{8t}}
(r+w)
e^{-\frac{(r+w)^2}{4t}}\,dw, \nonumber
\end{align}
where the inequality uses Part~(i) and~\eqref{AnInequinProof2}, and \(C=C(T,\beta)>0\) is obtained by combining all the constants. The integral above can be bounded as
\[
\int_0^\infty
w(r+w)
e^{-\frac{(r+w)^2}{8t}}\,dw 
=
\int_r^\infty
(u-r)u e^{-\frac{u^2}{8t}}\,du 
\le
\int_r^\infty
u^2 e^{-\frac{u^2}{8t}}\,du
=
\int_r^\infty
u^2 e^{-\frac{u^2}{16t}}e^{-\frac{u^2}{16t}}\,du ,
\]
where the second equality uses the change of variables \(u=r+w\). Observe that the integral above is bounded by $ e^{-\frac{r^2}{16t}} \int_0^\infty u^2 e^{-\frac{u^2}{16t}}\,du = C t^{3/2} e^{-\frac{r^2}{16t}}$. Substituting these observations into the above bound for $\big| \nabla_x Q_t^\beta(x)\big|$ yields the desired bound.
\end{proof}

\subsection{Proof of Lemma~\ref{LemmaL2IncrementY}} \label{ProofLemmaL2IncrementY}
\begin{proof}

Using the elementary inequality \( |a-b|^2\le 2|a|^2+2|b|^2 \)
together with the bound \( |\mathbb Q_t^\beta(y)|\le |y| \), we obtain the inequality below.
\[
\int_{\mathbb R^3}
\mathlarger p_{0,t}^{T,\beta}(x,y)
\left|
\mathbb Q_{T-t}^{\beta}(y)-\mathbb Q_t^\beta(x)
\right|^2dy
\le
2\int_{\mathbb R^3}
\mathlarger p_{0,t}^{T,\beta}(x,y)|y|^2dy
+
2|x|^2 .
\]
Thus, it remains to show that the second moment under the transition density is finite. Using~\eqref{FirstTrans}, we obtain the first equality below.
\begin{align}
&\int_{\mathbb R^3}
\mathlarger p_{0,t}^{T,\beta}(x,y)|y|^2dy \nonumber \\
&=
\frac{1}{1+Q_T^\beta(x)}
\int_{\mathbb R^3}
|y|^2
\big(1+Q_{T-t}^{\beta}(y)\big)
P_t^\beta(x,y)\,dy \nonumber \\
&\le
\int_{\mathbb R^3}
|y|^2
\big(1+Q_{T-t}^{\beta}(y)\big)
P_t(x,y)\,dy 
+
C\,t^{-1/2}\,\frac{e^{-\frac{|x|^2}{4t}}}{|x|}\,
\int_{\mathbb R^3}
|y|^2
\big(1+Q_{T-t}^{\beta}(y)\big)
\frac{e^{-\frac{|y|^2}{4t}}}{|y|}\,dy \nonumber \\
&\le
C\int_{\mathbb R^3}
|y|^2
\left(1+\frac{\sqrt T}{|y|}\right)
P_t(x,y)\,dy
+
C t^{-1/2}
\frac{e^{-\frac{|x|^2}{4t}}}{|x|}
\int_{\mathbb R^3}
|y|^2
\left(1+\frac{\sqrt T}{|y|}\right)
\frac{e^{-\frac{|y|^2}{4t}}}{|y|}\,dy . \nonumber
\end{align}
The first inequality follows from the fact that \(1+Q_T^\beta(x)\ge 1\) and the estimate~\eqref{PtBetaBounds}, whereas the second inequality holds by~\eqref{OnePlusQBetaBoundsOld} applied with \(L=\beta\sqrt T\). The first integral is also finite since $|y|^2(1+\sqrt T\,|y|^{-1}) = |y|^2+\sqrt T\,|y|$, and the three-dimensional free heat kernel \(P_t(x,\cdot)\) has finite moments of every order. The second integral is finite since $|y|^2(1+\sqrt T\,|y|^{-1})|y|^{-1} = |y|+\sqrt T$, and hence the integrand is bounded by $C(|y|+\sqrt T)e^{-\frac{|y|^2}{4t}}$, which is integrable over \(\mathbb R^3\).
\end{proof}

\subsection{Proof of Lemma~\ref{LemmaKbounds1}} \label{ProofLemmaKbounds1}

We begin by proving two auxiliary lemmas that will be used in the proof of Lemma~\ref{LemmaKbounds1}, which is given at the end of this section.

\begin{lemma} \label{LemFreeHeatWeightedGradient}
Fix \(T,\beta>0\). For every \(x\in\mathbb R^3\setminus\{0\}\),
\[
\int_0^T
\int_{\mathbb R^3}
P_s(x,y)
\frac{|\nabla_y Q_{T-s}^\beta(y)|^2}
     {(1+Q_{T-s}^\beta(y))^3}
\,dy\,ds
<\infty .
\]
\end{lemma}

\begin{proof}
By splitting the time integral, we may write the double integral above as the sum of the following two terms.
\begin{align}\label{TwoTermsFreeHeatWeightedGradient}
\int_0^{T/2}
\int_{\mathbb R^3}
P_s(x,y)
\frac{|\nabla_y Q_{T-s}^\beta(y)|^2}
     {(1+Q_{T-s}^\beta(y))^3}
\,dy\,ds
+
\int_{T/2}^{T}
\int_{\mathbb R^3}
P_s(x,y)
\frac{|\nabla_y Q_{T-s}^\beta(y)|^2}
     {(1+Q_{T-s}^\beta(y))^3}
\,dy\,ds 
\end{align}
Let us denote the two terms above by \(I_1\) and \(I_2\), respectively, and estimate each term separately. \vspace{.3cm}

\noindent \textit{First term $I_1$.} We split the spatial integral into two regions and, on the region \(|y|>1\), using the trivial bound \(1+Q_{T-s}^\beta(y)\ge1\), we obtain
\begin{align} 
I_1
\le
\int_0^{T/2}
\int_{\{|y|\le 1\}}
P_s(x,y)
\frac{|\nabla_y Q_{T-s}^\beta(y)|^2}
     {(1+Q_{T-s}^\beta(y))^3}
dy ds
+
\int_0^{T/2}
\int_{\{|y|>1\}}
P_s(x,y)
|\nabla_y Q_{T-s}^\beta(y)|^2
dy ds . \nonumber
\end{align}
We denote the two terms above by $I_1^{(1)}$ and $I_1^{(2)}$, respectively. \vspace{.2cm}

\noindent \textit{Estimate of \(I_1^{(1)}\).} Since \(0<s\le T/2\), we have \(T-s\in[T/2,T]\) and \(\beta\sqrt{T-s}\le\beta\sqrt T\). Hence, by~\eqref{GradQBetaBoundOld}, applied with \(L=\beta\sqrt{T}\) and \(T\) replaced by \(T-s\), there exists a constant \(C=C(T,\beta)>0\) such that
\begin{align} \label{GradBoundQinProof}
\big|\nabla_y Q_{T-s}^\beta(y)\big|^2
\le
C
\left(
\frac1{|y|}
+
\frac1{|y|^2}
\right)^2
e^{-8\pi\beta |y|}
\le
C
\left(
\frac1{|y|^2}
+
\frac1{|y|^3}
+
\frac1{|y|^4}
\right)
e^{-8\pi\beta |y|}.
\end{align}
On the other hand, by~\eqref{OnePlusQBetaBoundsOld}, again with \(T\) replaced by \(T-s\) and using \(T-s\in[T/2,T]\), we obtain $(1+Q_{T-s}^\beta(y))^{-3} \le C |y|^3 e^{\frac{3|y|^2}{4(T-s)}}$. Using this estimate together with the bound~\eqref{GradBoundQinProof}, we obtain
\begin{align} \label{GradQuotient}
\frac{
|\nabla_y Q_{T-s}^\beta(y)|^2
}{
(1+Q_{T-s}^\beta(y))^3
}
\le
C
\left(
|y|
+
1
+
\frac1{|y|}
\right)
e^{ -8\pi\beta |y|} e^{\frac{3|y|^2}{4(T-s)}}
\le
C
\left(
|y|+1+\frac1{|y|}
\right),
\end{align}
where the last inequality uses that since \(|y|\le1\) and \(T-s\in[T/2,T]\), the factor $e^{\frac{3|y|^2}{4(T-s)}}$ is bounded by a constant depending only on \(T\). Therefore,
\begin{align} \nonumber
I_1^{(1)}
&:=
\int_0^{T/2}
\int_{|y|\le1}
P_s(x,y)
\frac{|\nabla_y Q_{T-s}^\beta(y)|^2}
     {(1+Q_{T-s}^\beta(y))^3}
\,dy\,ds \nonumber \\
&\le
C
\int_0^{T/2}
\int_{|y|\le1}
P_s(x,y)
\left(
|y|+1+\frac1{|y|}
\right)
\,dy\,ds .
\end{align}
The right-hand side is finite because
$
\int_{|y|\le1}
\big(
|y|+1+|y|^{-1}
\big)
dy
=
4\pi\int_0^1
\left(
r^3+r^2+r
\right)\,dr
<\infty, $ and, for fixed \(x\neq0\), \(P_s(x,y)\) is locally integrable in \((s,y)\). \vspace{.2cm}

\noindent \textit{Estimate of \(I_1^{(2)}\).} Using the gradient estimate~\eqref{GradBoundQinProof}, we obtain the first inequality below.
\begin{align}
I_1^{(2)}
&:=\int_0^{T/2}
\int_{\{|y|>1\}}
P_s(x,y)
|\nabla_y Q_{T-s}^\beta(y)|^2
dy ds \nonumber \\
& \le
C
\int_0^{T/2}
\int_{|y|>1}
P_s(x,y)
\left(
\frac1{|y|^2}
+
\frac1{|y|^3}
+
\frac1{|y|^4}
\right)
e^{-8\pi\beta |y|}
dy ds
\le
C\int_0^{T/2}\int_{\mathbb R^3} P_s(x,y)
dy ds \nonumber
\end{align}
The final inequality follows from the bound
$\left( \frac1{|y|^2} + \frac1{|y|^3} + \frac1{|y|^4}
\right) e^{-8\pi\beta|y|} \le 3$ for \(|y|>1\), and an enlargement of the domain of integration. Since $\int_0^{T/2}\int_{\mathbb R^3} P_s(x,y)\,dy\,ds = \int_0^{T/2}1\,ds < \infty$, the term \(I_1^{(2)}\) is finite. Combining the estimates of \(I_1^{(1)}\) and \(I_1^{(2)}\), we conclude that \(I_1<\infty\). \vspace{.3cm}

\noindent \textit{Second term \(I_2\).} Since \(s\in[T/2,T]\), we have
$P_s(x,y)= (4\pi s)^{-3/2} e^{-\frac{|x-y|^2}{4s}} \le C$, where \(C=C(T)>0\). Therefore, the second term \(I_2\) in~\eqref{TwoTermsFreeHeatWeightedGradient} is bounded by
\[
I_2
\le
C
\int_{T/2}^{T}
\int_{\mathbb R^3}
\frac{|\nabla_y Q_{T-s}^\beta(y)|^2}
     {(1+Q_{T-s}^\beta(y))^3}
\,dy\,ds
=
C
\int_0^{T/2}
\int_{\mathbb R^3}
\frac{|\nabla_y Q_t^\beta(y)|^2}
     {(1+Q_t^\beta(y))^3}
\,dy\,dt ,
\]
where the equality uses the change of variables \(t=T-s\). Next, we split the \(y\)-integral into the two regions \(|y|\le \sqrt t\) and \(|y|>\sqrt t\) to obtain
\begin{align} \label{TwoTermsSplitI2}
I_2
&\le
C
\int_0^{T/2}
\int_{\{|y|\le\sqrt t\}}
\frac{|\nabla_y Q_t^\beta(y)|^2}
     {(1+Q_t^\beta(y))^3}
\,dy\,dt
+
C
\int_0^{T/2}
\int_{\{|y|>\sqrt t\}}
\frac{|\nabla_y Q_t^\beta(y)|^2}
     {(1+Q_t^\beta(y))^3}
\,dy\,dt \nonumber \\
&=:
I_2^{(1)} + I_2^{(2)}.
\end{align}

\noindent \textit{Estimate of \(I_2^{(1)}\).} Since \(0<t\le T/2\), we have \(\beta\sqrt t\le\beta\sqrt{T/2}\). Hence, by~\eqref{GradQBetaBoundOld}, applied with \(L=\beta\sqrt{T/2}\) and \(T\) replaced by \(t\), there exists a constant \(C=C(T,\beta)>0\) such that
\[
\big|\nabla_y Q_t^\beta(y)\big|^2
\le
C
\left[
\frac1{|y|}
+
\frac{1+\sqrt t}{|y|^2}
+
\frac1{|y|\sqrt t}
\right]^2
e^{-8\pi\beta|y|}
\le
\frac{C}{|y|^4},
\]
where the last inequality uses that \(|y|\le\sqrt t\le\sqrt{T/2}\), and hence \(|y|^{-1}\le C|y|^{-2}\) and \(\big(|y|\sqrt t\big)^{-1}\le |y|^{-2}\). Also, by~\eqref{OnePlusQBetaBoundsOld}, we have $(1+Q_t^\beta(y))^{-3} \le C|y|^3(\sqrt t)^{-3} e^{3|y|^2/(4t)}
\le Ct^{-3/2}|y|^3$, where the last inequality uses that \(|y|\le\sqrt t\). Using this estimate together with the bound above for $\big|\nabla_y Q_t^\beta(y)\big|^2$, we obtain
\[
\frac{
|\nabla_y Q_t^\beta(y)|^2
}{
(1+Q_t^\beta(y))^3
}
\le
C
\frac{|y|^3}{t^{3/2}}
\cdot
\frac1{|y|^4}
=
\frac{C}{|y|t^{3/2}},
\]
which implies the inequality below.
\begin{align}
I_2^{(1)}
:=
C\int_0^{T/2}
\int_{\{|y|\le\sqrt t\}}
\frac{|\nabla_y Q_t^\beta(y)|^2}
     {(1+Q_t^\beta(y))^3}
\,dy\,dt
&\le
C
\int_0^{T/2}
\frac1{t^{3/2}}
\int_0^{\sqrt t} r\,dr\,dt
=
C
\int_0^{T/2}
\frac1{\sqrt t}\,dt  \nonumber
\end{align}
The last integral is finite, and hence \(I_2^{(1)}<\infty\).  \vspace{.2cm}

\noindent \textit{Estimate of \(I_2^{(2)}\).} Using the trivial bound \(1+Q_t^\beta(y)\ge1\) and Lemma~\ref{LemSmallTimeGaussianEstimate}(ii), we obtain
\begin{align}\label{I22Bound}
\frac{|\nabla_y Q_t^\beta(y)|^2}{(1+Q_t^\beta(y))^3}
\le
|\nabla_y Q_t^\beta(y)|^2
\le
C
\left(
\frac1{|y|^2}
+
\frac{2\sqrt t}{|y|^3}
+
\frac{t}{|y|^4}
\right)
e^{-\frac{|y|^2}{8t}}
\le
\frac{C}{|y|^2}
e^{-\frac{|y|^2}{8t}} ,
\end{align}
where the final inequality uses that $2\sqrt t\,|y|^{-3} \le 2|y|^{-2}$ and \(t|y|^{-4}\le |y|^{-2}\) since  \(|y|>\sqrt t\). Therefore, the second term \(I_2^{(2)}\) in~\eqref{TwoTermsSplitI2} is bounded by
\begin{align}
I_2^{(2)}
&\le
C
\int_0^{T/2}
\int_{\{|y|>\sqrt t\}}
\frac1{|y|^2}
e^{-\frac{|y|^2}{8t}}
\,dy\,dt \nonumber \\
&=
C
\int_0^{T/2}
\int_{\sqrt t}^{\infty}
e^{-\frac{r^2}{8t}}
\,dr\,dt 
=
C
\int_0^{T/2}
\sqrt t
\int_1^\infty
e^{-\frac{u^2}{8}}
\,du\,dt
<\infty ,
\end{align}
where the first equality follows by changing into spherical coordinates, the second equality uses the change of variables \(r=\sqrt t\,u\), and the finiteness follows since
\(\int_1^\infty e^{-u^2/8}\,du<\infty\) and \(\int_0^{T/2}\sqrt t\,dt<\infty\).

Combining the estimates for \(I_2^{(1)}\) and \(I_2^{(2)}\), we get \(I_2<\infty\). Together with \(I_1<\infty\), this completes the proof.
\end{proof}

\begin{lemma} \label{LemSingularHeatWeightedGradient}
Fix \(T,\beta>0\). For every \(x\in\mathbb R^3\setminus\{0\}\),
\[
\int_0^T
s^{-1/2}
\frac{e^{-\frac{|x|^2}{4s}}}{|x|}
\int_{\mathbb R^3}
\frac{e^{-\frac{|y|^2}{4s}}}{|y|}
\frac{|\nabla_y Q_{T-s}^\beta(y)|^2}
     {(1+Q_{T-s}^\beta(y))^3}
\,dy\,ds
<\infty .
\]
\end{lemma}

\begin{proof}

We split the time integral according to \((0,T)=(0,T/2]\cup(T/2,T)\) and write the left-hand side above as the sum of two terms,
\begin{align}
\int_0^{T/2}
&s^{-1/2}
\frac{e^{-\frac{|x|^2}{4s}}}{|x|}
\int_{\mathbb R^3}
\frac{e^{-\frac{|y|^2}{4s}}}{|y|}
\frac{|\nabla_y Q_{T-s}^\beta(y)|^2}
     {(1+Q_{T-s}^\beta(y))^3}
\,dy\,ds \nonumber \\
&+
\int_{T/2}^{T}
s^{-1/2}
\frac{e^{-\frac{|x|^2}{4s}}}{|x|}
\int_{\mathbb R^3}
\frac{e^{-\frac{|y|^2}{4s}}}{|y|}
\frac{|\nabla_y Q_{T-s}^\beta(y)|^2}
     {(1+Q_{T-s}^\beta(y))^3}
\,dy\,ds , \nonumber
\end{align}
and denote the two terms by \(\mathbf I_1\) and \(\mathbf I_2\), respectively. \vspace{.2cm}

\noindent \textit{First term \(\mathbf I_1\).} We split the
\(y\)-integral into the regions \(|y|\le1\) and \(|y|>1\), and write
\(\mathbf I_1 = \mathbf I_1^{(1)} + \mathbf I_1^{(2)}\). We estimate these two terms separately below. \vspace{.2cm}

\noindent \textit{Estimate of $\mathbf I_1^{(1)}$.} Using the estimate~\eqref{GradQuotient}, we obtain the first inequality below.
\begin{align}
\mathbf I_1^{(1)}
&:=
\int_0^{T/2}
s^{-1/2}
\frac{e^{-\frac{|x|^2}{4s}}}{|x|}
\int_{\{|y|\le1\}}
\frac{e^{-\frac{|y|^2}{4s}}}{|y|}
\frac{|\nabla_y Q_{T-s}^\beta(y)|^2}
     {(1+Q_{T-s}^\beta(y))^3}
\,dy\,ds
\nonumber\\
&\le
C
\int_0^{T/2}
s^{-1/2}
\frac{e^{-\frac{|x|^2}{4s}}}{|x|}
\int_{\{|y|\le1\}}
\frac{e^{-\frac{|y|^2}{4s}}}{|y|}
\left(|y|+1+\frac1{|y|}\right)
\,dy\,ds  \nonumber \\
&\le
\frac{C}{|x|}
\int_0^{T/2}
s^{-1/2}
\int_{\{|y|\le1\}}
\left(1+\frac1{|y|^2}\right)
\,dy\,ds
=
\frac{C}{|x|}
\int_0^{T/2}
s^{-1/2}
\int_0^1
\left(r^2+1\right)
\,dr\,ds
< \infty \nonumber
\end{align}
The second inequality uses that \(|y|^{-1}\le |y|^{-2}\), which holds since \(|y|\le1\), and the trivial bounds \(e^{-\frac{|x|^2}{4s}} \le 1\) and \(e^{-\frac{|y|^2}{4s}} \le 1\). The finiteness follows since \(x\neq0\) is fixed and both integrals over \(r\) and \(s\) are finite. \vspace{.2cm}

\noindent \textit{Estimate of $\mathbf I_1^{(2)}$.} Using \(1+Q_{T-s}^\beta(y)\ge1\) and the gradient estimate
\eqref{GradBoundQinProof}, we obtain
\begin{align}
\mathbf I_1^{(2)}
&:=
\int_0^{T/2}
s^{-1/2}
\frac{e^{-\frac{|x|^2}{4s}}}{|x|}
\int_{\{|y|>1\}}
\frac{e^{-\frac{|y|^2}{4s}}}{|y|}
\frac{|\nabla_y Q_{T-s}^\beta(y)|^2}
     {(1+Q_{T-s}^\beta(y))^3}
\,dy\,ds
\nonumber\\
&\le
C
\int_0^{T/2}
s^{-1/2}
\frac{e^{-\frac{|x|^2}{4s}}}{|x|}
\int_{\{|y|>1\}}
\frac{e^{-\frac{|y|^2}{4s}}}{|y|}
\left(
\frac1{|y|^2}
+
\frac1{|y|^3}
+
\frac1{|y|^4}
\right)
e^{-8\pi\beta |y|}
\,dy\,ds . \nonumber
\end{align}
Since \(|y|>1\), we have
$\frac1{|y|}
\left(
\frac1{|y|^2}
+
\frac1{|y|^3}
+
\frac1{|y|^4}
\right)
\le 3$,
and therefore, enlarging the domain of integration, we may bound the above by
\[
\frac{C}{|x|}
\int_0^{T/2}
s^{-1/2}
\int_{\R^3}
e^{-\frac{|y|^2}{4s}}
\,dy\,ds
=
\frac{C}{|x|}
\int_0^{T/2}
s
\,ds
<\infty ,
\]
where the equality uses that $\int_{\mathbb R^3} e^{-\frac{|y|^2}{4s}}\,dy = C s^{3/2}$, and the finiteness follows since \(x\neq0\) is fixed. Therefore, \(\mathbf I_1=\mathbf I_1^{(1)}+\mathbf I_1^{(2)}<\infty\). \vspace{.2cm}

\noindent \textit{Second term \(\mathbf I_2\).} Since \(s\in[T/2,T]\) and \(x\neq0\), we have $|x|^{-1}s^{-1/2}e^{-\frac{|x|^2}{4s}}\le C_{T,x}$, and hence, after the change of variables \(t=T-s\), the inequality below follows.
\begin{align}
\mathbf I_2
&:=
\int_{T/2}^{T}
s^{-1/2}
\frac{e^{-\frac{|x|^2}{4s}}}{|x|}
\int_{\mathbb R^3}
\frac{e^{-\frac{|y|^2}{4s}}}{|y|}
\frac{|\nabla_y Q_{T-s}^\beta(y)|^2}
     {(1+Q_{T-s}^\beta(y))^3}
\,dy\,ds  \nonumber \\
&\le
C_{T,x}
\int_0^{T/2}
\int_{\mathbb R^3}
\frac1{|y|}
\frac{|\nabla_y Q_t^\beta(y)|^2}
     {(1+Q_t^\beta(y))^3}
\,dy\,dt 
\end{align}
Next, we split the \(y\)-integral into the regions \(|y|\le\sqrt t\) and \(|y|>\sqrt t\), and denote the corresponding terms by $\mathbf I_2^{(1)}$ and $\mathbf I_2^{(2)}$, respectively, so that $\mathbf I_2=\mathbf I_2^{(1)}+\mathbf I_2^{(2)}$. \vspace{.2cm}

\noindent \textit{Estimate of $\mathbf I_2^{(1)}$.} By Lemma~\ref{LemSmallTimeGaussianEstimate}(ii), we have
\[
\big|\nabla_y Q_t^\beta(y)\big|^2
\le
C
\left(
\frac1{|y|^2}
+
\frac{2\sqrt t}{|y|^3}
+
\frac{t}{|y|^4}
\right)
e^{-\frac{|y|^2}{8t}}
\le
C\frac{t}{|y|^4},
\]
where the second inequality uses that on the region \(|y|\le\sqrt t\), we have \(t^{-1/2}\le |y|^{-1}\) and $|y|^{-2}\le t|y|^{-4}$. Moreover, by~\eqref{OnePlusQBetaBoundsOld}, we have $(1+Q_t^\beta(y))^{-3} \le Ct^{-3/2}|y|^3$ on \(|y|\le\sqrt t\). Consequently,
\[
\frac1{|y|}
\frac{|\nabla_y Q_t^\beta(y)|^2}
     {(1+Q_t^\beta(y))^3}
\le
\frac{C}{|y|^2\sqrt t},
\]
which implies the inequity below.
\begin{align}
\mathbf I_2^{(1)}
&:=
C_{T,x}\int_0^{T/2}
\int_{\{|y|\le\sqrt t\}}
\frac1{|y|}
\frac{|\nabla_y Q_t^\beta(y)|^2}
     {(1+Q_t^\beta(y))^3}
\,dy\,dt \nonumber \\
&\le
C
\int_0^{T/2}
\frac1{\sqrt t}
\int_0^{\sqrt t} dr\,dt
=
C
\int_0^{T/2} 1\,dt
<\infty, 
\end{align}
where $C=C(T,\beta,x)>0$. \vspace{.2cm}

\noindent \textit{Estimate of $\mathbf I_2^{(2)}$.} Using the estimate~\eqref{I22Bound}, we obtain the inequality below.
\[
\mathbf I_2^{(2)}
:=
C_{T,x}
\int_0^{T/2}
\int_{\{|y|>\sqrt t\}}
\frac1{|y|}
\frac{|\nabla_y Q_t^\beta(y)|^2}
     {(1+Q_t^\beta(y))^3}
\,dy\,dt 
\le
C
\int_0^{T/2}
\int_{\{|y|>\sqrt t\}}
\frac1{|y|^3}
e^{-\frac{|y|^2}{8t}} \,dy\,dt .
\]
By changing to spherical coordinates and using the change of variables \(r=\sqrt t\,u\), we may write the right side above as
\[
C
\int_0^{T/2}
\int_{\sqrt t}^{\infty}
\frac1r
e^{-\frac{r^2}{8t}}
\,dr\,dt 
=
C
\int_0^{T/2}
\int_1^\infty
\frac1u
e^{-\frac{u^2}{8}}
\,du\,dt
<\infty .
\]
The finiteness follows since $\int_1^\infty u^{-1}e^{-\frac{u^2}{8}}\,du<\infty$ and $T<\infty$. Therefore, $\mathbf I_2^{(2)}<\infty$. Consequently, $\mathbf I_2=\mathbf I_2^{(1)}+\mathbf I_2^{(2)}<\infty$.
\end{proof}

\begin{proof}[Proof of Lemma~\ref{LemmaKbounds1}]
For fixed \(x\in\mathbb R^3\), recalling~\eqref{FirstTrans}, \eqref{DefDriftFun}, and~\eqref{DefpFunction}, we obtain
\begin{align}
\int_0^T
\int_{\mathbb R^3}
 \mathlarger{p}_{0,s}^{T,\beta}(x,y)
\,
\bigl|
 S_{T-s}^{\beta}(y)
b_{T-s}^{\beta}(y)
\bigr|^2
\,dy\,ds
\leq
\int_0^T
\int_{\mathbb R^3}
\,P_s^\beta(x,y)
\,
\frac{
|\nabla_y Q_{T-s}^\beta(y)|^2
}{
(1+Q_{T-s}^\beta(y))^3
}
\,dy\,ds , \nonumber 
\end{align}
where the inequality also uses the trivial bound $1+Q_T^\beta(x)\ge1$. Next, using the estimate~\eqref{PtBetaBounds} for the point interaction heat kernel, we may bound the right-hand side above by
\[
\int_0^T
\int_{\mathbb R^3}
P_s(x,y)
\frac{|\nabla_y Q_{T-s}^\beta(y)|^2}
     {(1+Q_{T-s}^\beta(y))^3}
\,dy\,ds
+
C
\int_0^T
s^{-1/2}
\frac{e^{-\frac{|x|^2}{4s}}}{|x|}
\int_{\mathbb R^3}
\frac{e^{-\frac{|y|^2}{4s}}}{|y|}
\frac{|\nabla_y Q_{T-s}^\beta(y)|^2}
     {(1+Q_{T-s}^\beta(y))^3}
\,dy\,ds .
\]
Both of the above terms are finite by Lemma~\ref{LemFreeHeatWeightedGradient} and Lemma~\ref{LemSingularHeatWeightedGradient}, respectively, which completes the proof.
\end{proof}

\section*{Acknowledgement}
The author thanks Yu Gu for drawing attention to the articles
\cite{CranstonKoralovMolchanovVainberg1,CranstonKoralovMolchanovVainberg2}. The author is also grateful to Makoto Nakashima and Rongfeng Sun for helpful discussions related to this work.

\end{document}